\documentclass[reqno,english,11pt]{amsart}
\usepackage{amsfonts,amsmath,latexsym,verbatim,amscd,mathrsfs,color,array}

\usepackage{amsmath,amssymb,amsthm,amsfonts}
\usepackage{graphicx,color}

\usepackage{cite}

 \usepackage{graphicx}

\newtheorem{lemma}{Lemma}[section]

\newtheorem{proposition}{Proposition}[section]

\newtheorem{remark}{Remark}[section]
\newtheorem{theorem}{Theorem}[section]
\newcommand{\ve} {\varepsilon}
\newcommand{\R}{{\mathbb R}}

\newcommand{\ep}{\epsilon}
\newcommand{\Om}{\Omega}
\def\bs{\begin{split}}
\def\es{\end{split}}

\newcommand{\vect}[2]{\ensuremath{\left( \begin{array}{c}
            #1 \\
            #2
            \end{array}
        \right)}}

\numberwithin{equation}{section}

\title[Boundary Spike Cluster]{Stable Boundary Spike Clusters for the Two-Dimensional Gierer-Meinhardt System}

\author[W. Ao]{Weiwei Ao }
\address{\noindent Wuhan University, Department of Mathematics, China, 430072}
\email{wwao@whu.edu.cn}

\author[J. Wei]{Juncheng Wei}
\address{\noindent
Department of Mathematics,
University of British Columbia, Vancouver, B.C., Canada, V6T 1Z2}
\email{jcwei@math.ubc.ca}

\author[M. Winter]{Matthias Winter}
\address{Brunel University London, Department of Mathematics, Uxbridge UB8 3PH, United Kingdom}
\email{matthias.winter@brunel.ac.uk}

\usepackage{pstricks}
\usepackage{color}

\newrgbcolor{LemonChiffon}{1. 0.98 0.8}
\newrgbcolor{LightBlue}{0.68 0.85 0.9}
\newrgbcolor{darkgreen}{0.0 0.45 0.0}
\newrgbcolor{brightgreen}{0.0 0.85 0.0}
\newrgbcolor{verylightgray}{0.9 0.9 0.9}
\newrgbcolor{gray}{0.4 0.4 0.4}
\newrgbcolor{Red}{1 0 0}
\newrgbcolor{Blue}{0 0 1}

\begin{document}

\begin{abstract}
We consider the Gierer-Meinhardt system with small inhibitor diffusivity and very small activator diffusivity in a bounded and smooth two-dimensional domain.
For any given positive integer $k$
we construct a spike cluster consisting of $k$ boundary spikes which all approach the same nondegenerate local maximum point of the
boundary curvature. We show that this spike cluster is linearly stable.

The main idea underpinning
these stable spike clusters
is the following: due to the small inhibitor diffusivity
the  interaction between spikes
is repulsive and the spikes are attracted towards a nondegenerate local maximum point of the boundary curvature.
Combining these two effects can lead to an equilibrium of spike positions within the cluster such that the cluster is linearly stable.

\end{abstract}

\maketitle


\section{Introduction}
\setcounter{equation}{0}

Turing in his pioneering work in 1952 \cite{T} proposed that a patterned distribution of two chemical substances, called the morphogens, could trigger the emergence of cell structures. He also gives the following explanation for the formation of the morphogenetic pattern: It is assumed that one of the morphogens, the activator, diffuses slowly and the other, the inhibitor, diffuses much faster. In the mathematical framework of a coupled system of reaction-diffusion equations with hugely different diffusion coefficients it is shown by linear stability analysis that the homogeneous state may be unstable. In particular, a small perturbation of spatially homogeneous initial data may evolve to a stable spatially complex pattern of the morphogens.

Since the work of Turing, many different reaction-diffusion system in biological modelling have been proposed and the occurrence of pattern formation has been investigated by studying what is now called Turing instability. One of the most popular models in biological pattern formation is the Gierer-Meinhardt system \cite{GM}, see also \cite{M}. In two dimensions in a special case after rescaling it can be stated as follows:
\begin{equation}
\label{gm}
\left\{\begin{array}{l}
u_t=\ve^2\Delta u-u+\frac{u^2}{v}, u>0 \mbox{ in }\Omega,\\[2mm]
\tau v_t=D\Delta v-v+u^2, v>0 \mbox{ in }\Omega,\\[2mm]
\frac{\partial u}{\partial \nu}=\frac{\partial v}{\partial \nu}=0 \mbox{ on }\partial \Omega.
\end{array}
\right.
\end{equation}
The unknown functions $u=u(x,t)$ and $v=v(x,t)$ represent the concentrations of the activator and inhibitor, respectively, at the point $x\in\Omega\subset \R^2$ and at a time $t>0$. Here $\Delta $ is the Laplace operator in $\R^2$, $\Omega$ is a smooth bounded domain in $\R^2$, $\nu=\nu(x)$ is the outer unit normal at $x\in \partial \Omega$.

Throughout this paper, we assume that
\begin{equation}
\label{smalld}
0<\ve\ll1, \ 0<D\ll1,
\end{equation}
$\tau\geq 0$ is a fixed constant independent of $\ve$, $D$ and $x$.
Further, the diffusivities $\ve$ and $D$ do not depend on $x$ but they are both small constants.
In this paper, we further assume that
\begin{equation}\label{equal}
e^{-\frac{1}{\sqrt{D}}}\ll\ve\ll\sqrt{D}.
\end{equation}
This means that $\ep$ is much smaller than ${D}$. On the other hand, $\ep$ cannot be exponentially small compared to $\sqrt{D}$.

In this paper, we study the Gierer-Meinhardt system in a bounded and smooth two-dimensional domain. We prove the existence and stability of a cluster consisting of $k$ boundary spikes near a nondegenerate local maximum point $P^0$ of the boundary curvature $h(P)$.

The main idea underpinning
these stable spike clusters
is the following: due to the small inhibitor diffusivity
the  interaction between spikes
is repulsive and the spikes are attracted towards a nondegenerate local maximum point of the boundary curvature. Combining these two effects can lead to an equilibrium of spike positions within the cluster such that the cluster is linearly stable.

The repulsive nature of spikes has been shown in \cite{ei-wei-2002}. Existence and stability of a spike cluster made up of two boundary spikes has been established in \cite{ei-ishimoto-2013}.

Before we state our main results, let us mention some previous ones concerning various regimes for the asymptotic behavior of $D$.

For the strong coupling case, i.e. $D\sim 1$, the second and third authors constructed single-interior spike solutions \cite{WW2}. In \cite{WW5}, they continued the study, and proved the existence of solutions with $k$ interior spikes.

Moreover, it is shown that this solution is linearly stable for $\tau=0$.

For the weak coupling case $D\to\infty$, in \cite{WW4} the second and third authors proved the existence of multiple interior spike solutions.
Further, they showed that there are stability thresholds
\begin{equation*}
D_1(\ve)>D_2(\ve)>\cdots>D_k(\ve)>\cdots
\end{equation*}
such that
if $\lim_{\ve\to 0}\frac{D_k(\ve)}{D}>1$, the $k$-peak solution is stable and if $\lim_{\ve\to 0}\frac{D_k(\ve)}{D}<1$, the $k$-peak solution is unstable.
Multiple spikes for the Gierer-Meinhardt system in a one-dimensional interval have been studied in \cite{T1,IWW,WW3} and on the real line in \cite{DGK}.

In \cite{W5} the existence, uniqueness and spectral properties of a boundary spike solution have been studied for the shadow Gierer-Meinhardt system (i.e. after formally taking the limit $D\to\infty$.)

In \cite{WW6} the existence and stability of $N-$peaked steady states for the Gierer-Meinhardt system with precursor inhomogeneity has been explored.
The spikes in the patterns can vary in amplitude. In particular, the results imply that a precursor inhomogeneity can induce instability.
Single-spike solutions for the Gierer-Meinhardt system with precursor including spike dynamics have been studied in \cite{wmhgm}.

Previous results on stable spike clusters include a stable spike cluster for a consumer chain model \cite{WW7} and a stable spike cluster for the one-dimensional Gierer-Meinhardt system with precursor inhomogeneity \cite{WW8}.

For more background, modelling, analysis and computation on the Gierer-Meinhardt system, we refer to \cite{ww-book} and the references therein.

\section{Main Results: Existence and Stability}
\setcounter{equation}{0}

Let $\Omega\subset \R^2$ be bounded and smooth two-dimensional domain.
Let $w$ be the unique solution in $H^1(\R^2)$ of the problem
\begin{equation}\label{ground}
\left\{\begin{array}{l}
\Delta w-w+w^2=0,\\[2mm]
w(0)=\max_{y\in \R^2}w(y), \ w(y)\to 0 \mbox{ as }|y|\to \infty.
\end{array}
\right.
\end{equation}
For the existence and uniqueness of (\ref{ground}), we refer to \cite{K} and \cite{NT}. We also recall that $w$ is radially symmetric and
\begin{equation*}
w(y)\sim |y|^{-\frac{1}{2}}e^{-|y|} \quad\mbox{ as }|y|\to \infty
\end{equation*}
and
\begin{equation*}
w'(y)=-(1+o(1))w(y) \quad\mbox{ as }|y|\to \infty,
\end{equation*}
where $w'$ is the radial derivative of $w$, i.e. $w'=w_r(r)$.

\begin{theorem}
\label{existenceth}
Let $k$ be a positive integer, and $P^0$ be a nondegenerate local maximum point of the curvature $h(P)$ of the boundary $\partial\Om$.

Then for $0<e^{-\frac{1}{\sqrt{D}}}\ll\ve\ll\sqrt{D}\ll1$, the Gierer-Meinhardt system (\ref{gm}) has a $k$-boundary spike cluster steady-state solution $(u_\ep,v_\ep)$ which concentrates near $P^0$. In particular, it satisfies
\[
u_\ep\sim \frac{D\xi_\sigma}{\ep^2} \sum_{i=1}^k w(\frac{x-P_{i,\ep}}{\ep}),
\]
where $P_{i,\ep}\to P^0$ as $\ep\to 0$ for $i=1,2,\ldots,k$.

Further, we have
\[
\xi_\sigma\sim \frac{1}{\log\frac{\sqrt{D}}{\ep}}
\]
and
\[
|P_i-P_{i-1}|\sim \sqrt{D}\log \frac{\xi_\sigma}{\ep D},\quad i=2,\ldots,k.
\]
\end{theorem}

\begin{remark}
The spike cluster is established by a balance between repelling spikes and attracting boundary point of local maximum curvature.
\end{remark}

\begin{theorem}
\label{stabilityth}
The $k$-boundary spike cluster solution given in Theorem \ref{existenceth} is linearly stable if $\tau$ is small enough.
\end{theorem}

\begin{remark}
There are eigenvalues of two different orders:
$n-1$ eigenvalue related to repelling of neighbouring spikes are of order $\ep^3\log \frac{\xi_\sigma}{\ep D}$, and one eigenvalue stemming from the curvature of the boundary
(corresponding to synchronous motion of all spikes) is of order $\ep^3$.
\end{remark}

This paper is organised as follows. In sections 3-5 we show existence of the spike cluster steady state by using Liapunov-Schmidt reduction. In section 3 we introduce an approximation to the spike cluster steady state.  In section 4 we use the Liapunov-Schmidt method to reduce the problem to finite dimensions. In section 5 we solve this reduced problem. In sections 6-7 we study the stability of this spike cluster steady state. In section 6 we consider large eigenvalues. Finally, in section 7 we study small eigenvalues.
In two appendices we show some technical results: in appendix A (section 8) we prove Proposition 4.1 and in appendix B (section 9) we compute the small eigenvalues.

\medskip

\section{Introduction of the approximate solutions}
\setcounter{equation}{0}

In this paper, we consider the following problem:
\begin{equation}\label{e:1}
\left\{\begin{array}{l}
\ve^2\Delta u-u+\frac{u^2}{v}=0 \mbox{ in }\Omega,\\[2mm]
D\Delta v-v+u^2=0 \mbox{ in }\Omega,\\[2mm]
\frac{\partial u}{\partial \nu}=\frac{\partial v}{\partial \nu}=0 \mbox{ on }\partial \Omega,
\end{array}
\right.
\end{equation}
where $\Om\subset \R^2$ is a bounded and smooth two-dimensional domain.

Let $P^0$ be a nondegenerate maximum point of the boundary curvature $h(P)$ on the boundary  of $\Omega$.
 For $P\in \partial\Omega$,
\begin{equation*}
\nabla_{\tau(P)}:=\frac{\partial }{\partial \tau(P)},
\end{equation*}
where $\frac{\partial }{\partial \tau(P)}$ denotes the tangential derivative with respect to $P$ at $P\in \partial \Omega$. We will sometimes drop the variable $P$ if this can be done without causing confusion.

In this section, we construct an approximation to a spike cluster solution  to (\ref{e:1}) which concentrates at $P^0$.

The approximate cluster consists of spikes $\sigma^{-2}\xi_{\sigma,i}w(\frac{x-P_i}{\ve})$ which are centred at the points $P_i$ for $i=1,\cdots,k$, where $\sigma=\frac{\ve}{\sqrt{D}}$ and the amplitude $\xi_{\sigma,i}$ satisfies
\[
\xi_{\sigma,i}\sim \frac{1}{
\frac{1}{\pi}\log\frac{1}{\sigma}
\int_{\R_+^2} w^2\,dx}
\]
(see (\ref{e:7})).

Let $P_1,\cdots,P_k$ be $k$ points distributed along the boundary $\partial \Omega$ such that we have for $i=2,\cdots,k$
\begin{eqnarray}\label{e:2}
\left|\frac{P_i-P_{i-1}}{\sqrt{D}}-\log\frac{\xi_\sigma}{\ve D}+\frac{3}{2}\log\log\frac{\xi_\sigma}{\ve D}+\log(-\frac{\frac{\partial^2}{\partial \tau^2}h(P^0)\nu_1}{2\nu_2})
+\log[(i-1)(k+1-i)]\right|\leq \eta\nonumber\\[2mm]
\end{eqnarray}
and
\begin{equation}\label{e:3}
\left|\frac{1}{k}\sum_{i=1}^kP_i-P^0\right|\leq \eta \sqrt{D} \log \frac{\xi_\sigma}{\ve D},
\end{equation}
where $\nu_2$ is given in (\ref{nu2}) below. Further, $\eta>0$ is a small constant independent of $\ve$ and $D$.
The reason for assuming (\ref{e:2}) and (\ref{e:3}) will become clear in Section 5 when we solve the reduced problem.

\begin{remark}
By (\ref{e:2}), the  distance of neighbouring spikes satisfies
 \[
 |P_i-P_{i-1}|\sim C \sqrt{D}\log \frac{\xi_\sigma}{\ve D}.\]
  Since we want to construct multiple boundary spikes which collapse at one point, we require that assumptions (\ref{smalld}) and (\ref{equal}) hold.
\end{remark}
After re-scaling,
\begin{equation*}
\hat{u}(z)=\sigma^2u(\ve z),
\quad \hat{v}(z)=\sigma^2v(\ve z),
\end{equation*}
if we drop the hat and still denote solutions by $(u,v)$, equation (\ref{e:1}) is equivalent to
\begin{equation}\label{e:4}
\left\{\begin{array}{l}
\Delta u-u+\frac{u^2}{v}=0 \mbox{ in }\Omega_\ve,\\[2mm]
\Delta v-\sigma^2v+u^2=0 \mbox{ in }\Omega_\ve,\\[2mm]
\frac{\partial u}{\partial \nu}=\frac{\partial v}{\partial \nu}=0 \mbox{ on }\partial \Omega_\ve,
\end{array}
\right.
\end{equation}
where $\Omega_\ve=\ve^{-1}\Omega$.

\medskip

From now on we will deal with (\ref{e:4}). Before introducing an approximation to the spike solutions, we first define some notation.

Fixing ${\bf P}^\ve=(P_1,\cdots,P_k)$ such that (\ref{e:2}) and (\ref{e:3}) hold, we set
\begin{equation}\label{e:5}
\Lambda_k=\{{\bf p}^{\ve}=\ve^{-1}{\bf P}^\ve\,: \,P_1,\cdots,P_k \ \mbox{ such that} \ (\ref{e:2}) \mbox{ and } \ (\ref{e:3}) \ \mbox { hold }\}.
\end{equation}

We are looking for multiple spike solutions to (\ref{e:4}) of the form
\begin{equation}\label{e:6}
\left\{\begin{array}{l}
u(z)\sim \sum_{i=1}^k\xi_{\sigma,i}Pw_{p_i}(z-p_i),\\[2mm]
v(p_i)\sim \xi_{\sigma,i},
\end{array}
\right.
\end{equation}
where $Pw_{p_i}(z-p_i)$ is defined to be the unique solution of
\begin{equation}
\label{proj1}
\Delta u-u+w(\cdot-p_i)^2=0 \mbox{ in }\Omega_{\ve,p_i}, \ \frac{\partial u}{\partial \nu}=0 \mbox{ on }\partial \Omega_{\ve,p_i}.
\end{equation}
Here
$\Omega_{\ve,p_i}=\{z\,:\, z+p_i\in\Omega_\ve\}$,
the function $w$ has been defined in (\ref{ground})
 and $\xi_{\sigma,i}, \ i=1,\cdots,k$ are the heights of the spikes, which will be determined in (\ref{e:7}).

\subsection{The analysis of the projection $Pw_{q}(z-q)$}

Before calculating the heights of the spikes, we need some preliminaries of the projection $Pw_{p_i}(z-p_i)$ defined in (\ref{proj1}) which are rather standard by now. Some of these results have been derived in \cite{WW,WW1}.

Let $P\in \partial \Omega$. We define a diffeomorphism straightening the boundary. We may assume that the inward normal to $\partial \Omega$ at $P$ is pointing in the direction of the positive $x_2$ axis. Denote $B'({R})=\{x\in \R^2||x_1|\leq {R}\}$.
 Then since $\partial \Omega$ is smooth, we can find a constant $R$ such that $\partial \Omega$ can be represented by the graph of a smooth function $\rho_P:B'(R)\to \R$, where $\rho_P(0)=0$, and $\rho_P'(0)=0$. From now on, we omit the use of $P$ in $\rho_P$ and write $\rho$ if this can be done without causing confusion. So near $P$, $\partial \Omega$ can be represented by $(x_1, \rho(x_1))$. The curvature of $\partial \Omega$ at $p$ is
$h(P)=\rho''(0)$.
 Let $\Omega_{1}=\Omega\cap B(P,
{R})=\{(x_1,x_2)\in B(P,{R})| x_2-P_2>\rho(x_1-P_1)\}$, where $B(P,
R)=\{x\in \R^2| |x-P|<R\}$.

After rescaling, it follows that near $p=\frac{P}{\ve}$, the boundary $\partial \Omega_\ve$ can be represented by $(z_1-{\bf p}_1, \ve^{-1}\rho(\ve (z_1-{\bf p}_1)))$, where $(z_1,z_2)=\ve^{-1}(x_1,x_2)$ and $p=({\bf p}_1, {\bf p}_2)$.
By Taylor expansion, we have
\begin{equation}
\label{taylor}
\ve^{-1}\rho(\ve (z_1-{\bf p}_1))=\frac{1}{2}\rho''(0)\ve (z_1-{\bf p}_1)^2+\frac{1}{6}\rho^{(3)}(0)\ve^2 (z_1-{\bf p}_1)^3+O(\ve^3 (z_1-{\bf p}_1)^4).
\end{equation}

Let $h_p(z)=w(z-p)-Pw_p(z-p)$. Then $h_p$ satisfies
\begin{equation}
\left\{\begin{array}{l}
\Delta h_p(z)-h_p(z)=0, \mbox{ in }\Omega_\ve,\\[2mm]
\frac{\partial h_p}{\partial \nu}=\frac{\partial }{\partial \nu}w(z-p) \mbox{ on }\partial \Omega_\ve.
\end{array}
\right.
\end{equation}
For $z\in \Omega_{1\ve}=\frac{1}{\ve}\Omega_1$ we set
\begin{equation}\label{y}
\left\{\begin{array}{l}
y_1=z_1-{\bf p}_1,\\[2mm]
y_2=z_2-{\bf p}_2-\ve^{-1}\rho(\ve (z_1-{\bf p}_1)).
\end{array}
\right.
\end{equation}

Under this transformation, the Laplace operator and the boundary derivative operator become
\begin{eqnarray*}
&&\Delta _z=\Delta_y+(\rho'(\ve y_1))^2\partial_{y_2y_2}-2\rho'(\ve y_1)\partial_{y_1y_2}-\ve\rho''(\ve y_1)\partial_{y_2},\\[2mm]
&&(1+\rho'(\ve y_1)^2)^{\frac{1}{2}}\frac{\partial }{\partial \nu}=\rho'(\ve y_1)\partial_{y_1}-(1+{\rho'}(\ve y_1)^2)\partial_{y_2}.
\end{eqnarray*}

Let $v^{(1)}$ be the unique solution of
\begin{equation}
\left\{\begin{array}{l}
\Delta v-v=0 \mbox{ in }\R^2_+,\\[2mm]
\frac{\partial v}{\partial y_2}=\frac{w'}{|y|}\frac{\rho''(0)}{2}y_1^2 \mbox{ on }\partial \R^2_+,
\end{array}
\right.
\end{equation}
where $\R^2_+$ is the upper half plane, namely
$\R^2_+=\{y=(y_1,y_2)\in \R^2 | y_2>0\}.$

 Let $v^{(2)}$ be the unique solution of
\begin{equation*}
\left\{\begin{array}{l}
\Delta v-v-2\rho''(0)y_1\frac{\partial^2 v^{(1)}}{\partial_{y_1}\partial_{y_2}}=0  \mbox{ in }\R^2_+,\\[2mm]
\frac{\partial v}{\partial y_2}=-\rho''(0)y_1\frac{\partial v^{(1)}}{\partial y_1} \mbox{ on }\partial \R^2_+.
\end{array}
\right.
\end{equation*}
Let $v^{(3)}$ be the unique solution of
\begin{equation}
\label{v3equ}
\left\{\begin{array}{l}
\Delta v-v=0\mbox{ in }\R^2_+,\\[2mm]
\frac{\partial v}{\partial y_2}=\frac{w'}{|y|}\frac{1}{3}\rho^{(3)}(0)y_1^3
 \mbox{ on }\partial \R^2_+.
\end{array}
\right.
\end{equation}

Note that $v^{(1)}, v^{(2)}$ are even functions in $y_1$ and $v^{(3)}$ is an odd function in $y_1$. Moreover, it is easy to see that $|v_i(y)|\leq Ce^{-\mu|y|}$ for any $0<\mu<1$. Let $\chi$ be a smooth cut-off function, such that $\chi(a)=1$ for $a\in B(0,R_0\sqrt{D}\log\frac{\xi_\sigma}{\ve D})$, and $\chi(a)=0$ for $x\in B(0,2R_0\sqrt{D}\log\frac{\xi_\sigma}{\ve D})^c$
for some
suitable $R_0$ such that $|p_i-p^0|<\frac{R_0}{\sigma}
\log\frac{\xi_\sigma}{\ep D}$, and \begin{equation}
\label{chive}
\chi_\ve(z-p)=\chi(\ve (z-p))
\quad \mbox{ for }
z\in \Omega_\ve.
\end{equation}
Set
\begin{equation}
h_p(z)=-(\ve v^{(1)}(y)+\ve^2(v^{(2)}(y)+v^{(3)}(y)))\chi_\ve(z-p)+\ve^3\xi_p(z), \ z\in \Omega_\ve.
\end{equation}
Then we have the following estimate:
\begin{proposition}\label{pro1}
\begin{equation}
\|\xi_p(z)\|_{H^1(\Omega_\ve)}\leq C.
\end{equation}
\end{proposition}

Proposition \ref{pro1} was proved in \cite{WW1} by Taylor expansion including a rigorous bound for the remainder using estimates for elliptic partial differential equations. Moreover, it has been shown that that $|\xi_p(z)|\leq Ce^{-\mu|z-p|}$ for any $0<\mu<1$.

Similarly we know from \cite{WW1} that
\begin{proposition}\label{pro2}
\begin{equation}
[\frac{\partial w}{\partial \tau(p)}-\frac{\partial p_{\Omega_\ve}w}{\partial \tau(p)}](z-p)=\ve \eta(y)\chi_\ve(z-p)+\ve^2\eta_1(z), \ z\in \Omega_\ve,
\end{equation}
where $\eta$ is the unique solution of the following equation:
\begin{equation}
\left\{\begin{array}{l}
\Delta \eta-\eta=0 \mbox{ in }\R^2_+,\\[2mm]
\frac{\partial \eta}{\partial y_2}=-\frac{1}{2}(\frac{w''}{|y|^2}-\frac{w'}{|y|^3})\rho''(0)y_1^3-\frac{w'}{|y|}\rho''(0)y_1 \mbox{ on }\partial R^2_+.
\end{array}
\right.
\end{equation}
Moreover,
\begin{equation}
\|\eta_1\|_{H^1(\Omega_\ve)}\leq C.
\end{equation}
\end{proposition}
It follows that $\eta(y)$ is an odd function in $y_1$. It can be seen that $|\eta_1(y)|\leq Ce^{-\mu|y|}$ for some $0<\mu<1$.

Finally, let
\begin{equation}
L_0=\Delta-1+2w(z).
\end{equation}
We have
\begin{lemma}\label{lem1}
\begin{equation}
Ker(L_0)\cap H^2_N(\R^2_+)=span\{\frac{\partial w}{\partial y_1}\},
\end{equation}
where $H^2_N(\R^2_+)=\{u\in H^2(\R^2_+)\,:\,\frac{\partial u}{\partial y_2}=0 \mbox{ on }\partial \R^2_+\}$.
\end{lemma}
\begin{proof}
See Lemma 4.2 in \cite{NT1}.
\end{proof}

\begin{remark}
In the following sections, we will denote by $y^i=(y^i_1,y^i_2)$ the transformation defined by (\ref{y}) centred at the point $p_i$ and let $v_i^{(j)}$ be the corresponding solutions in the expansion of $h_{p_i}$.
\end{remark}

\subsection{The analysis of the Green's function}
Next we introduce a Green's function $G_{\sqrt{D}}$ which is needed to derive our main results.

For $D>0$, let $G_D(x,y)$ be the Green's function given by
\begin{equation}
\left\{\begin{array}{l}
-\Delta G_{\sqrt{D}}+G_{\sqrt{D}}=\delta_y \mbox{ in }\Omega_{\sqrt{D}},\\[2mm]
\frac{\partial G_{\sqrt{D}}}{\partial \nu}=0 \mbox{ on }\partial \Omega_{\sqrt{D}},
\end{array}
\right.
\end{equation}
where $y\in \partial \Omega_{\sqrt{D}}$ and let $G_0$ be the Green's function of the upper half plane:
\begin{equation}
-\Delta G_0+G_0=\delta_0 \mbox{ in }\R^2_+,\\[2mm]
\frac{\partial G_0}{\partial y_2}=0 \mbox{ on }\partial \R^2_+.
\end{equation}

Then $H(x)=G_{\sqrt{D}}(x)-G_0(x)$ will satisfy
\begin{equation}
\left\{\begin{array}{l}
\Delta H-H=0 \mbox{ in }\Omega_{\sqrt{D}}, \\[2mm]
\frac{\partial H}{\partial \nu}=-\frac{\partial G_0}{\partial \nu} \mbox{ on }\partial \Omega_{\sqrt{D}}.
\end{array}
\right.
\end{equation}
Let $\eta_1$ be the solution of
\begin{equation}
\left\{\begin{array}{l}
\Delta \eta_1-\eta_1=0 \mbox{ in }\R^2_+,\\[2mm]
\frac{\partial \eta_1}{\partial y_2}=-\frac{1}{2}\sqrt{D}\frac{G_0'(|y|)}{|y|}\rho''(0)y_1^2,
\end{array}
\right.
\end{equation}
and let $\eta_2$ be the solution of
\begin{equation}
\left\{\begin{array}{l}
\Delta \eta_2-\eta_2=0 \mbox{ in }\R^2_+,\\[2mm]
\frac{\partial \eta_2}{\partial y_2}=D(-\frac{1}{3}\frac{G'_0(|y|)}{|y|}\rho^{(3)}(0)y_1^3)\mbox{ on }\partial \R^2_+.
\end{array}
\right.
\end{equation}
 It can be seen easily that $\eta_1$ is even in $y_1$ and $\eta_2$ is odd in $y_1$.
Then one can get the following result.
\begin{lemma}\label{lemma3}
\begin{equation}
G_{\sqrt{D}}(x,p)=G_0(x,p)+\sqrt{D}\eta_1(y)\chi_{\sqrt{D}}(x-p)+D\eta_2(y)\chi_{\sqrt{D}}(x-p)+O(D^{\frac{3}{2}}).
\end{equation}
\end{lemma}
\begin{proof}
First we compute on $\partial \Omega_{\sqrt{D}}$,
\begin{eqnarray*}
&&\sqrt{1+\rho'(\sqrt{D}x_1)^2}\frac{\partial }{\partial \nu}G_0(x)\\[2mm]
&&=\frac{G_0'(|x|)}{|x|}(x_2-x_1\rho'(\sqrt{D}x_1))\\[2mm]
&&=\frac{G_0'(|x|)}{|x|}(-\frac{1}{2}\sqrt{D}\rho''(0)y_1^2-\frac{1}{3}D\rho^{(3)}(0)y_1^3)+O(D^{\frac{3}{2}}e^{-l|y|}),
\end{eqnarray*}
for any $0<l<1$.

Since we have
\begin{eqnarray*}
&&\frac{G_0'(|x|)}{|x|}=\frac{G_0'(|y|)}{|y|}+D\frac{|y|G_0''(|y|)-G_0'(|y|)}{8|y|^3}(\rho''(0)y_1^2)^2+O(D^{\frac{3}{2}}e^{-l|y|}),
\end{eqnarray*}
we get
\begin{eqnarray*}
&&\sqrt{1+\rho'(\sqrt{D}x_1)^2}\frac{\partial }{\partial \nu}G_0(x)\\[2mm]
&&=-\frac{1}{2}\sqrt{D}\frac{G_0'(|y|)}{|y|}\rho''(0)y_1^2-\frac{1}{3}D\frac{G_0'(|y|)}{|y|}\rho^{(3)}(0)y_1^3+O(D^{\frac{3}{2}}e^{-l|y|}).
\end{eqnarray*}
From the expansion above, we can get the asymptotic behaviour of $G_{\sqrt{D}}$.

\end{proof}
Next we have the following expansion of $G_0$:
\begin{lemma}
The following expansion of $G_0$ holds:
\begin{equation*}
G_0(r)=-\frac{1}{\pi}\log r+c_1+c_2r^2\log r+\psi(r),
\end{equation*}
for $0<r<1$, where $\psi$ is a smooth function with $\psi(0)=\psi'(0)=0$ and $c_1, c_2$ are universal constants.
\end{lemma}
\begin{proof}
By an even extension in $y_2$, one can get the Green's function in the whole space $\R^2$. For the expansion of fundamental solution, see Lemma 4.1 in \cite{DKW}. Then the above expansion follows.

\end{proof}

We set
\begin{equation}
G_{\sqrt{D}}(x,y)=\frac{1}{\pi}\log\frac{1}{|x-y|}+\tilde{H}(x,y).
\end{equation}

From the estimates above, and for points ${\bf p}^\ve\in \Lambda_k$, we have
\begin{equation}\label{e:10}
G_{\sqrt{D}}(\sigma p_i,\sigma p_j)=O(\frac{\ve D}{\xi_\sigma}\log \frac{\ve D}{\xi_\sigma}) \mbox{ for } |i-j|=1,
\end{equation}
\begin{equation}
G_{\sqrt{D}}(\sigma p_i,\sigma p_j)=O((\frac{\ve D}{\xi_\sigma}\log \frac{\ve D}{\xi_\sigma})^2)\mbox{ for } |i-j|= 2.
\end{equation}
Generally, we have
\begin{equation}
G_{\sqrt{D}}(\sigma p_i,\sigma p_j)=O((\frac{\ve D}{\xi_\sigma}\log \frac{\ve D}{\xi_\sigma})^{|i-j|})\mbox{ for } |i-j|\geq 1.
\end{equation}
For the derivatives, we estimate
\begin{equation}\label{e:11}
\frac{\partial^l }{\partial p_i^l}G_{\sqrt{D}}(\sigma p_i,\sigma p_j)=O((\frac{\ve D}{\xi_\sigma}\log \frac{\ve D}{\xi_\sigma})^{|i-j|}\sigma^{l})\mbox{ for } |i-j|\geq 1.
\end{equation}

\subsection{Calculating the heights of the spikes}
In this section, we are going to determine the heights of spikes $\xi_{\sigma,i}$ to leading order. In the sequel, by $T[h]$ we denote the unique solution of the equation
\begin{equation}
\left\{\begin{array}{l}
\Delta v-\sigma^2v+h=0 \mbox{ in }\Omega_\ve,\\[2mm]
\frac{\partial v}{\partial \nu}=0 \mbox{ on }\partial \Omega_\ve.
\end{array}
\right.
\end{equation}
Then we know that
\begin{equation}
v(z)=\int_{\Omega_\ve}G_{\sqrt{D}}(\sigma z,\sigma x)h(x)\,dx.
\end{equation}
As mentioned before, we will choose the approximate solution to be
\begin{equation}
U(z)=\sum_{i=1}^k \xi_{\sigma,i}Pw_{p_i}(z-p_i)
\end{equation}
and
\begin{eqnarray}
V(z)&&=T[U^2](z)\\
&&=\int_{\Omega_\ve}G_{\sqrt{D}}
(\sigma z,\sigma x)
\left(\sum_{i=1}^k \xi_{\sigma,i}Pw_{p_i}(x-p_i)\right)^2\,dx.
\nonumber
\end{eqnarray}
First we calculate the heights of the peaks:
\begin{eqnarray*}
V(p_i)&&=\int_{\Omega_\ve}G_{\sqrt{D}}(\sigma p_i,\sigma x)\left(\sum_{i=1}^k\xi_{\sigma,i}Pw_{p_i}(x-p_i)\right)^2\,dx\\
&&=\int_{\Omega_\ve}G_{\sqrt{D}}(\sigma p_i,\sigma x)\left(\sum_{i=1}^k\xi_{\sigma,i}^2
\left(Pw_{p_i}(x-p_i)\right)^2\right)\,dx+O(\ve^4)\\
&&=\xi_{\sigma,i}^2\int_{\Omega_\ve}
\left(\frac{1}{\pi}\log\frac{1}{\sigma|x-p_i|}+\tilde{H}(\sigma x,\sigma p_i)\right)(Pw_{p_i}(x-p_i))^2\,dx\\
&&+\sum_{j\neq i}\xi_{\sigma,j}^2\int_{\Omega_\ve}G_{\sqrt{D}}(\sigma x,\sigma p_i)(Pw_{p_j}(x-p_j))^2\,dx+O(\ve^4)\\
&&=(\frac{1}{\pi}\log\frac{1}{\sigma}\int_{\R^2_+}w^2\,dx)\xi_{\sigma,i}^2+\sum_{j=1}^kO(\xi_{\sigma,j}^2).
\end{eqnarray*}

Thus
\begin{equation}
\xi_{\sigma,i}=(\frac{1}{\pi}\log\frac{1}{\sigma}\int_{\R^2_+}w^2\,dx)\xi_{\sigma,i}^2+\sum_{j=1}^kO(\xi_{\sigma,j}^2).
\end{equation}
We assume that the heights of the spikes are asymptotically equal
as $\ve, D\to 0$, i.e.
\begin{equation}
\lim_{\sigma\to 0}\frac{\xi_{\sigma,i}}{\xi_{\sigma,j}}=1, \mbox{ for }i\neq j.
\end{equation}
Then we get that
\begin{eqnarray}\label{e:7}
\xi_{\sigma,i}&&=
(\frac{1}{\pi}\log\frac{1}{\sigma}
\int_{\R^2_+}w^2\,dx)^{-1}(1+O(\frac{1}{\log \frac{1}{\sigma}}))\nonumber\\
&&=\xi_{\sigma}(1+O(\frac{1}{\log \frac{1}{\sigma}})),
\end{eqnarray}
where
\begin{eqnarray}\label{e:7a}
\xi_{\sigma}=(\frac{1}{\pi}\log\frac{1}{\sigma}
\int_{\R^2_+}w^2\,dx)^{-1}.
\end{eqnarray}

The analysis in this subsection calculates the heights of the spikes under the assumption that their shape is given. In the next two sections, we provide the rigorous proof for the existence.

\section{Existence I: Reduction to finite dimensions}
\setcounter{equation}{0}

Let us start to prove Theorem \ref{existenceth}.

The first step is choosing a good approximate solution which was done in (\ref{e:6}). The second step is using the Liapunov-Schmidt method to reduce the problem to a finite dimensions which we do in this section. The last step is  solving the reduced problem which will be done in Section 5.

First we need to calculate the error terms caused by the approximate solution given in (\ref{e:6}) to show that this is a good choice:\
\begin{eqnarray*}
S_1(U,V)&&=\Delta U-U+\frac{U^2}{V}\\
&&=\frac{U^2}{V}-\sum_{i=1}^k\xi_{\sigma,i}w(x-p_i)^2\\
&&=\frac{\sum_{i=1}^k\xi_{\sigma,i}^2Pw_{p_i}(x-p_i)^2}{V}-\sum_{i=1}^k\xi_{\sigma,i}w(x-p_i)^2+O(\ve^4)\\
&&=\sum_{i=1}^k\xi_{\sigma,i}(Pw_i^2-w_i^2)+\sum_{i=1}^k\xi_{\sigma,i}^2Pw_i^2(\frac{1}{V(x)}-\frac{1}{V(p_i)})+O(\ve^4),
\end{eqnarray*}
where we have used the notation
\begin{equation}
Pw_i(x)=Pw_{p_i}(x-p_i), \ w_i(x)=w(x-p_i).
\end{equation}
On the other hand, we calculate for $x=p_i+z$
\begin{eqnarray*}
Pw_i(x)^2-w_i^2(x)&=&2w(z)(\ve v_i^{(1)}(z)+\ve^2v_i^{(2)}(z)+\ve^2 v_i^{(3)})(z)+\ve^2(v_i^{(1)}(z))^2+O(\ve^3)\\
&:=&\ve R_{1,i}(z)+\ve^2R_{2,i}(z)+O(\ve^3),
\end{eqnarray*}
where $R_{1,i}(z)=2w(z)[v_i^{(1)}(z)+\ve v_i^{(2)}(z)]+\ve (v_{1}^{(i)}(z))^2$, $R_{2,i}=2w(z)v_i^{(3)}(z)$.
This implies
\begin{eqnarray*}
&&V(p_i+z)-V(p_i)\\
&&=\int_{\Omega_\ve}[G_{\sqrt{D}}(\sigma p_i+\sigma z,\sigma x)-G_{\sqrt{D}}(\sigma p_i,\sigma x)](\sum_{i=1}^k\xi_{\sigma,i}Pw_i)^2\,dx\\
&&=\int_{\Omega_\ve}\frac{1}{\pi}\log\frac{|x-p_i|}{|x-z-p_i|}Pw_i^2\xi_{\sigma,i}^2+(\tilde{H}(\sigma p_i+\sigma y,\sigma x)-\tilde{H}(\sigma p_i,\sigma x))Pw_i^2\xi_{\sigma,i}^2\,dx\\
&&+\sum_{j\neq i}\xi_{\sigma,j}^2\int_{\Omega_\ve}[G_{\sqrt{D}}(\sigma p_i+\sigma z,\sigma x)-G_{\sqrt{D}}(\sigma z,\sigma x)]Pw_j^2\,dx+O(\ve^4)\\
&&=\xi_{\sigma,i}^2\int_{\R^2_+}\log\frac{|y|}{|y-z|}w^2(y)(1+o(1))\,dy+\sum_{j\neq i}\xi_{\sigma,j}^2\nabla_{p_i} G_{\sqrt{D}}(\sigma p_i,\sigma p_j)\cdot  z\int_{\R^2_+}w^2\,dy\\
&&+O(\xi_\sigma^2\ve^2h'(\ve p_i))+O(\xi_\sigma^2\ve^3)
+O(\sum_{j\neq i}\xi_{\sigma}^2|\nabla^2_{p_i} G_{\sqrt{D}}(\sigma p_i,\sigma p_j)|)\\
&&:=\xi_{\sigma,i}^2R_1(z)+\xi_{\sigma,i}^2R_2(z)+h.o.t,
\end{eqnarray*}
where $R_1(z)$ is even in $z_1$ and $R_2(z)$ is odd in $z_1$ and
\begin{equation}
R_1(z)=O(\log(1+|z|)), \ R_2(z)=O(\sum_{j\neq i}|\nabla_{p_i} G_{\sqrt{D}}(\sigma p_i,\sigma p_j)||z|)
\end{equation}
and
\begin{eqnarray}
h.o.t&&=O(\xi^2_\sigma \ve^2h'(\ve p_i))+O(\sum_{i\not=j}\xi_\sigma^2|\nabla^2_{p_i} G_{\sqrt{D}}(\sigma p_i,\sigma p_j)|)\\
&&=O(\xi_\sigma^2\ve^2\sqrt{D}\log\frac{\xi_\sigma}{\ve D}).\nonumber
\end{eqnarray}
Thus we can get that
\begin{equation}
\label{4.4}
\frac{1}{V(p_i+z)}-\frac{1}{V(p_i)}=\frac{1}{V(p_i)^2}(-\xi_{\sigma,i}^2R_1(z)-\xi_{\sigma,i}^2R_2(z)+h.o.t).
\end{equation}
By the above estimates, we have the following key estimate:
\begin{lemma}\label{lemma1}
For $x=p_i+z$, $|z|<\frac{R_0}{\sigma}\log \frac{\ve D}{\xi_\sigma}$, we have
\begin{equation}
S_1(U,V)=S_{1,1}+S_{1,2},
\end{equation}
where
\begin{equation}
S_{1,1}=\xi_{\sigma,i}^2\tilde{R}_1(z)
\end{equation}
\begin{equation}
S_{1,2}=-\xi_{\sigma,i}^2w^2(z)R_2(z)+\xi_{\sigma,i}\ve^2 R_{2,i}(z)+h.o.t,
\end{equation}
where $\tilde{R}_1(z)$ is even in $z_1$ with the property that $\tilde{R}_1(z)=O(\log(1+|z|))$, and $R_2(z), R_{2,i}(z)$ are defined above.

Further, $S_1(U,V)=\ve^{\frac{R_0}{\sigma}}$ for $|x-p_i|\geq \frac{R_0}{\sigma}\log \frac{\ve D}{\xi_\sigma}$ for all $i$.

\end{lemma}
The above estimates will be very useful in the existence proof using the Liapunov-Schmidt reduction. In particular, they will imply an explicit formula for the positions of the spikes in Section 5.

\medskip

Now we study the linearised operator defined by
\begin{equation}
L_{\ve,{\bf p}}:=S'\vect{U}{V},
\end{equation}
\begin{equation}
L_{\ve, {\bf p}}: H_N^2(\Omega_\ve)\times H_N^2(\Omega_\ve)\to L^2(\Omega_\ve)\times L^2(\Omega_\ve).
\end{equation}
We first define
\begin{equation}
K_{\ve,{\bf p}}=C_{\ve,{\bf p}}=\operatorname{Span}\{\frac{\partial U}{\partial \tau(p_i)},i=1,\cdots,k\}
\end{equation}
and define the approximate kernels by
\begin{equation*}
\mathcal{K}_{\ve,{\bf p}}:=K_{\ve,{\bf p}}+\{0\}\subset H_N^2(\Omega_\ve)\times H_N^2(\Omega_\ve),
\end{equation*}
and choose the approximate cokernels as follows:
\begin{equation*}
\mathcal{C}_{\ve,{\bf p}}:=C_{\ve,{\bf p}}+\{0\}\subset L^2(\Omega_\ve)\times L^2(\Omega_\ve).
\end{equation*}
We then define
\begin{eqnarray*}
&&\mathcal{K}_{\ve,{\bf p}}^\perp:=K_{\ve,{\bf p}}^\perp+H_N^2(\Omega_\ve)\subset H_N^2(\Omega_\ve)\times H_N^2(\Omega_\ve),\\
&&\mathcal{C}_{\ve,{\bf p}}^\perp:=C_{\ve,{\bf p}}^\perp+L^2(\Omega_\ve)\subset L^2(\Omega_\ve)\times L^2(\Omega_\ve),
\end{eqnarray*}
where $C_{\ve,{\bf p}}^\perp, K_{\ve,{\bf p}}^\perp$ denote the orthogonal complements with the scalar product of $L^2(\Omega_\ve)$.

Let $\pi_{\ve,{\bf p}}$ denote the projection in $L^2(\Omega_\ve)\times L^2(\Omega_\ve)$ onto $\mathcal{C}_{\ve,{\bf p}}^\perp$. We are going to show that the equation

\begin{equation*}
\pi_{\ve,{\bf p}}\circ S_{\ve}\vect{U+\phi}{V+\psi}=0
\end{equation*}
has a unique solution $\Sigma_{\ve,{\bf p}}=\vect{\phi_{\ve,{\bf p}}}{\psi_{\ve,{\bf p}}}\in \mathcal{K}_{\ve,{\bf p}}^\perp$ if $\max\{\sigma, D\}$ is small enough.

Set
\begin{equation}
\mathcal{L}_{\ve,{\bf p}}=\pi_{\ve,{\bf p}}\circ L_{\ve,{\bf p}}: \mathcal{K}_{\ve,{\bf p}}^\perp\to \mathcal{C}_{\ve,{\bf p}}^\perp.
\end{equation}
Written in components, we have
\[
\mathcal{L}_{\ve,{\bf p}}:=
\left(
\begin{array}{l}
\mathcal{L}_{\ve,{\bf p},1}\\[1mm]
\mathcal{L}_{\ve,{\bf p},2}
\end{array}
\right)
\]
and
\[
\left(
\begin{array}{l}
\mathcal{L}_{\ve,{\bf p},1}\\[1mm]
\mathcal{L}_{\ve,{\bf p},2}
\end{array}
\right)
\vect{\phi_{\ve,{\bf p}}}{\psi_{\ve,{\bf p}}}
=
\left(
\begin{array}{l}
\Delta \phi_{\ve,{\bf p}} - \phi_{\ve,{\bf p}}
+\frac{2U}{V}\phi_{\ve,{\bf p}}-\frac{U^2}{V^2}\psi_{\ve,{\bf p}}
\\[1mm]
\Delta\psi_{\ve,{\bf p}}-\psi_{\ve,{\bf p}}
+2U\phi_{\ve,{\bf p}}
\end{array}
\right).
\]
As a preparation we state a result on the invertibility of the corresponding linearised operator $\mathcal{L}_{\ve,{\bf p}}$ whose proof is postponed to Appendix A.

\begin{proposition}
\label{linboun}
There exist positive constants $\bar{\delta},C$ such that for $\max\{\sigma,D\}<\bar{\delta}$, the map $\mathcal{L}_{\ve,{\bf p}}$ is surjective for arbitrary ${\bf p}\in \Lambda_k$. Moreover the following estimate holds:
\begin{equation}
\|\Sigma_{\ve,{\bf p}}\|_{H^{2}(\Omega_\ve)\times H^2(\Omega_\ve)}\leq C(\|\mathcal{L}_{\ve,{\bf p},1}(\Sigma_{\ve,{\bf p}})\|_{H^2(\Omega_\ve)}+\xi_{\sigma}^{-1}\|\mathcal{L}_{\ve,{\bf p},2}(\Sigma_{\ve,{\bf p}})\|_{H^2(\Omega_\ve)}).
\end{equation}
\end{proposition}

Now we are in the position to solve the equation
\begin{equation}
\pi_{\ve,{\bf p}}\circ S_{\ve}\vect{U+\phi}{V+\psi}=0.
\end{equation}

Since $\mathcal{L}_{\ve,{\bf p}}|_{\mathcal{K}_{\ve,{\bf p}}^\perp}$ is invertible, we can write the above equation as
\begin{equation}
\Sigma=-\mathcal{L}_{\ve,{\bf p}}^{-1}\circ \pi_{\ve,{\bf p}}(S_\ve\vect{U}{V})-\mathcal{L}_{\ve,{\bf p}}^{-1}\circ \pi_{\ve,{\bf p}}(N_{\ve,{\bf p}}(\Sigma)):=M_{\ve,{\bf p}}(\Sigma),
\end{equation}
where
\begin{equation*}
\Sigma=\vect{\phi}{\psi}
\end{equation*}
and
\begin{equation*}
N_{\ve,{\bf p}}(\Sigma)=S_\ve\vect{U+\phi}{V+\psi}-S_\ve\vect{U}{V}-S_{\ve}'\vect{U}{V}\vect{\phi}{\psi}
\end{equation*}
and the operator $M_{\ve,{\bf p}}$ is defined above for $\Sigma \in H_N^2(\Omega_\ve)\times H^2_N(\Omega_\ve)$. We are going to show that the operator $M_{\ve,{\bf p}}$ is a contraction mapping on
\begin{equation}
B_{\ve}=\{\Sigma\in H_N^2(\Omega_\ve)\times H_N^2(\Omega_\ve)| \|\Sigma\|_{H^2\times H^2}<C_0\xi_{\sigma}^2\}
\end{equation}
if $C_0$ is large enough. We have that
\begin{eqnarray*}
\|M_{\ve,{\bf p}}(\Sigma)\|_{H^2\times H^2}&\leq &C(\|\pi_{\ve,{\bf p}}\circ N_{\ve,{\bf p},1}(\Sigma)\|_{L^2}+\xi_{\sigma}^{-1}\|\pi_{\ve,{\bf p}}\circ N_{\ve,{\bf p},2}(\Sigma)\|_{L^2}\\
&+&\|\pi_{\ve,{\bf p}}\circ S_\ve\vect{U}{V}\|_{L^2\times L^2}\\
&\leq &C(c(\xi_{\sigma})\xi_{\sigma}^2+\xi_\sigma^2),
\end{eqnarray*}
where $C>0$ is independent of $\ve>0$ and $c(\xi_\sigma)\to 0$ as $\xi_\sigma \to0$. Similarly we can show that
\begin{equation*}
\|M_{\ve,{\bf p}}(\Sigma)-M_{\ve,{\bf p}}(\Sigma')\|_{H^2\times H^2}\leq Cc(\xi_\sigma)\|\Sigma-\Sigma'\|_{H^2\times H^2},
\end{equation*}
where $c(\xi_\sigma)\to 0$ as $\xi_\sigma \to 0$. If we choose $C_0$ large enough, then $M_{\ve,{\bf p}} $ is a contraction mapping on $B_{\ve}$. The existence of a fixed point $\Sigma_{\ve,{\bf p}}$ together with an error estimate now follows from the contraction mapping principle. Moreover, $\Sigma_{\ve,{\bf p}}$ is a solution. Thus we have proved

 \begin{lemma}\label{lemma2}
 There exists $\bar{\delta}>0$ such that for every triple $(\ve, D, {\bf p})$ with $\max\{\sigma, D\}<\bar{\delta}$, and ${\bf p} \in\Lambda_k$, there exists a unique $(\phi_{\ve,{\bf p}},\psi_{\ve,{\bf p}})\in\mathcal{K}_{\ve,{\bf p}}^\perp$ satisfying
 \begin{equation}\label{e:8}
 S_\ve\vect{U+\phi_{\ve,{\bf p}}}{V+\psi_{\ve,{\bf p}}}\in \mathcal{C}_{\ve,{\bf p}},
 \end{equation}
 and
 \begin{equation*}
 \|(\phi_{\ve,{\bf p}},\psi_{\ve,{\bf p}})\|_{H^2\times H^2}\leq C\xi_\sigma^2.
 \end{equation*}
 \end{lemma}

 More refined estimates for $\phi_{\ve,{\bf p}}$ are needed. We recall that from Lemma \ref{lemma1} that $S_1(U,V)$ can be decomposed into two parts $S_{1,1}$ and $S_{1,2}$, where $S_{1,1}$ is in leading order  an even function in $z_1$ and $S_{1,2}$ is in leading order an odd function in $z_1$. Similarly we can decompose $\phi_{\ve,{\bf p}}$.
 \begin{lemma}\label{lemma3a}
 Let $\phi_{\ve,{\bf p}}$ be defined  by (\ref{e:8}). Then for $x=p_i+z$, we have
 \begin{equation}
 \phi_{\ve,{\bf p}}(x)=\phi_{\ve,{\bf p},1}+\phi_{\ve,{\bf p},2},
 \end{equation}
 where $\phi_{\ve,{\bf  p},1}$ is even in $z_1$ which can be estimated by
 \begin{equation}
 \phi_{\ve,{\bf p},1}=O(\xi_\sigma^2) \mbox{ in }H_N^2(\Omega_\ve),
 \end{equation}
 and $\phi_{\ve,{\bf p},2}$ can be estimated by
 \begin{eqnarray}
 \phi_{\ve,{\bf p},2}&&=O(\sum_{j\neq i}\xi^2_\sigma \sigma |\nabla G_{\sqrt{D}}(\sigma p_i,\sigma p_j)|)+O(\sum_{i=1}^k\xi_{\sigma}\ve^2h'(\ve p_i))\\
 &&=O(\xi_\sigma \ve^2\sqrt{D}\log\frac{\xi_\sigma}{\ve D}).\nonumber
 \end{eqnarray}
 \end{lemma}

\begin{proof}
Let $S(u)=S_1(u,T(u))$. We first solve
\begin{equation}
S(U+\phi_{\ve,{\bf p},1})-S(U)+\sum_{i=1}^k S_{1,1}(x-p_i)\in C_{\ve,{\bf p}},
\end{equation}
for $\phi_{\ve,{\bf p},1}\in K_{\ve,{\bf p}}^\perp$.
Then we solve
\begin{eqnarray*}
S[U+\phi_{\ve,{\bf p},1}+\phi_{\ve,{\bf p},2}]-S[U+\phi_{\ve,{\bf p},1}]+\sum_{i=1}^kS_{1,2}(x-p_i)\in C_{\ve,{\bf p}}
\end{eqnarray*}
for $\phi_{\ve,{\bf p},2}\in K_{\ve,{\bf p}}^\perp$.

Using the same proof as in Lemma \ref{lemma2}, the above two equations are uniquely solvable for $\max\{\sigma,D\}\ll1$. By uniqueness, $\phi_{\ve,{\bf p}}=\phi_{\ve,{\bf  p},1}+\phi_{\ve,{\bf p},2}$. Since $S_{1,2}=S_{1,2}^0+S_{1,2}^\perp$, where $S_{1,2}^0=O(\xi_\sigma \ve^2\sqrt{D}\log\frac{\xi_\sigma}{\ve D})$ and $S_{1,2}^\perp\in C_{\ve,{\bf p}}^\perp$, it is easy to see that $\phi_{\ve,{\bf p},1}$ and $\phi_{\ve,{\bf p},2}$ have the required properties.

\end{proof}

 \section{ Existence proof II: The reduced problem}
 \setcounter{equation}{0}

 In this section, we solve the reduced problem. This completes the proof for our main existence result given in Theorem \ref{existenceth}.

 By Lemma \ref{lemma2}, for every ${\bf p}\in \Lambda_k$, there exists a unique solution $(\phi,\psi)\in \mathcal{K}_{\ve,{\bf p}}^\perp$ such that
 \begin{equation}
  S_\ve\vect{U+\phi}{V+\psi}\in \mathcal{C}_{\ve,{\bf p}}.
 \end{equation}

 We need to determine ${\bf p}=(p_1,\cdots,p_k)\in \Lambda_k$ such that
 \begin{equation*}
 S_\ve\vect{U+\phi}{V+\psi}\perp \mathcal{C}_{\ve,p}
 \end{equation*}
 and therefore $S_\ve\vect{U+\phi}{V+\psi}=0$.

 \medskip

 To this end, we calculate the projection:
 \begin{eqnarray*}
 &&\int_{\Omega_\ve}S_1(U+\phi,V+\psi)\frac{\partial Pw_i}{\partial \tau(p_i)}\,dx\\
 &&=\int_{\Omega_\ve}\Big( \Delta(U+\phi)-(U+\phi)+\frac{(U+\phi)^2}{V+\psi} \Big)\frac{\partial Pw_i}{\partial \tau(p_i)}\,dx\\
 &&=\int_{\Omega_\ve}\Big( \Delta(U+\phi)-(U+\phi)+\frac{(U+\phi)^2}{V} \Big)\frac{\partial Pw_i}{\partial \tau(p_i)}\,dx\\
 &&+\int_{\Omega_\ve}\Big(\frac{(U+\phi)^2}{V+\psi} -\frac{(U+\phi)^2}{V}\Big)\frac{\partial Pw_i}{\partial \tau(p_i)}\,dx\\
 &&=I_1+I_2,
 \end{eqnarray*}
 where $I_1,I_2$ are defined by the last equality.

 For $I_1$, we have
 \begin{eqnarray*}
 I_1&=&\int_{\Omega_\ve}\Big( \Delta(U+\phi)-(U+\phi)+\frac{(U+\phi)^2}{V} \Big)\frac{\partial Pw_i}{\partial \tau(p_i)}\,dx\\
&=&\int_{\Omega_\ve}\Big(\Delta(\xi_{\sigma,i}Pw_i+\phi)-(\xi_{\sigma,i}Pw_i+\phi)+\frac{(\xi_{\sigma,i}Pw_i+\phi)^2}{V(p_i)}\Big)\frac{\partial Pw_i}{\partial \tau(p_i)}\,dx\\
&+&\int_{\Omega_\ve}(\xi_{\sigma,i}Pw_i+\phi)^2(\frac{1}{V(x)}-\frac{1}{V(p_i)})\frac{\partial Pw_i}{\partial \tau(p_i)}\,dx+O(\ve^4)\\
&=&I_{11}+I_{12}+O(\ve^4).
 \end{eqnarray*}

 Note that by the estimates satisfied by $\phi$ in Lemma \ref{lemma3a}, we have
 \begin{eqnarray}\label{e:15}
 &&\int_{\Omega_\ve}(\Delta \phi-\phi+2Pw_i\phi)\frac{\partial Pw_i}{\partial \tau(p_i)}\,dx\nonumber\\
 &=&\int_{\Omega_\ve}\phi\frac{\partial }{\partial \tau(p_i)}(Pw_i^2-w_i^2)\nonumber\\
 &=&\int_{\Omega_\ve}(\phi_{\ve,{\bf p},1}+\phi_{\ve,{\bf p},2})\frac{\partial }{\partial \tau(p_i)}[2\ve w_iv_i^{(1)}+2\ve^2w_iv_i^{(2)}+2\ve^2w_iv_i^{(3)}+\ve^2(v_i^{(1)})^2]\,dx+O(\xi_\sigma^2\ve^3)\nonumber\\
 &=&\sum_{j\neq i}O(\ve\xi_\sigma^2|\nabla_{p_i} G_{\sqrt{D}}(\sigma p_i,\sigma p_j)|)+\sum_{i=1}^k\xi_\sigma^2\ve^2|h'(\ve p_i)|+O(\ve^3\xi_\sigma^2)\nonumber\\
 &=&O(\xi_\sigma)[\sum_{j\neq i}\xi_\sigma^2\sigma |\nabla G_{\sqrt{D}}(\sigma p_i,\sigma p_j)|+\sum_{i=1}^k\xi_\sigma \ve^2 |h'(\ve p_i)|]\nonumber\\
 &=&O(\xi_\sigma^2\sigma \ve D\log\frac{\xi_\sigma}{\ve D}),
 \end{eqnarray}
 where we have used the estimates (\ref{e:10})-(\ref{e:11}).
 Further, we have
 \begin{eqnarray}\label{e:16}
 \int_{\Omega_\ve}\frac{\phi^2}{V(p_i)}\frac{\partial Pw_i}{\partial \tau(p_i)}\,dx&=&\frac{1}{\xi_{\sigma,i}}
 \int_{\Omega_\ve}(\phi_{\ve,{\bf p},1}+\phi_{\ve,{\bf p},2})^2\frac{\partial Pw_i}{\partial \tau(p_i)}\,dx\nonumber\\
 &=&O(\xi_\sigma)(\sum_{j\neq i}\xi_\sigma^2\sigma |\nabla G_{\sqrt{D}}(\sigma p_i,\sigma p_j)|+\sum_{i=1}^k\xi_\sigma \ve^2 |h'(\ve p_i)|)\nonumber\\
 &=&O(\xi_\sigma^2\sigma \ve D\log\frac{\xi_\sigma}{\ve D}).
 \end{eqnarray}
 We compute
 \begin{eqnarray}\label{e:17}
 &&\xi_{\sigma,i}\int_{\Omega_\ve}(\Delta Pw_i-Pw_i+Pw_i^2)\frac{\partial Pw_i}{\partial \tau(p_i)}\,dx\nonumber\\
 &&=\xi_{\sigma,i}\int_{\Omega_\ve}(Pw_i^2-w_i^2)\frac{\partial Pw_i}{\partial \tau(p_i)}\,dx\nonumber\\
 &&=\xi_{\sigma,i}\int_{\R^2_+}\ve(2 w v_i^{(1)}+2\ve w v_i^{(2)}+2\ve w v_i^{(3)}+\ve (v_i^{(1)})^2)
 \\ && \times(\frac{\partial w}{\partial y_1}+\ve\frac{\partial v_i^{(1)}}{\partial y_1}+\ve^2\frac{\partial v_i^{(2)}}{\partial y_1}+\ve^2\frac{\partial v_i^{(3)}}{\partial y_1})\,dy+O(\xi_\sigma \ve^3)\nonumber\\
 &&=2\ve^2\xi_{\sigma,i}\int_{\R^2_+}w v_i^{(3)}\frac{\partial w}{\partial y_1}\,dy+O(\xi_\sigma^2\sigma \ve D\log\frac{\xi_\sigma}{\ve D}).
 \end{eqnarray}
By (\ref{v3equ}), we have
\begin{eqnarray}\label{e:18}
\int_{\R^2_+}2w(y)\frac{\partial w(y)}{\partial y_1}v_i^{(3)}\,dy&=&\int_{\R^2_+}-(\Delta -1)\frac{\partial w(y)}{\partial y_1}v_i^{(3)}\,dy\nonumber\\
&=&-\int_{\R}\left(\frac{\partial w(y)}{\partial y_1}\frac{\partial v_i^{(3)}}{\partial y_2}-v_i^{(3)}\frac{\partial }{\partial y_2}\frac{\partial w(y)}{\partial y_1}\right)\,dy_1\nonumber\\
&=&-\frac{1}{3}\int_{\R}(\frac{w'(|y|)}{|y|})^2h'(\ve p_i)y_1^4\,dy_1\nonumber\\
&=&-\nu_1 \frac{\partial }{\partial \tau(\ve p_i)}h(\ve p_i),
\end{eqnarray}
where the constant $\nu_1>0$ is defined by
\begin{equation}
\nu_1=\frac{1}{3}\int_\R\left(\frac{\partial w(y_1,0)}{\partial y_1}\right)^2y_1^2\,dy_1>0.
\label{nu1}
\end{equation}

Now by (\ref{e:15})-(\ref{e:18}),
\begin{equation}\label{e:19}
I_{11}=-\ve^2\xi_\sigma \nu_1 \frac{\partial }{\partial \tau(\ve p_i)}h(\ve p_i)+O(\xi_\sigma^2\sigma \ve D\log\frac{\xi_\sigma}{\ve D}).
\end{equation}
 Next we estimate $I_{12}$:
 \begin{eqnarray}\label{e:20a}
 I_{12}&&=\int_{\Omega_\ve}(\xi_{\sigma,i}Pw_i+\phi)^2(\frac{1}{V(x)}-\frac{1}{V(p_i)})\frac{\partial Pw_i}{\partial \tau(p_i)}\,dx\nonumber\\
 &&=\int_{\Omega_\ve}(\xi_{\sigma,i}Pw_i+\phi_{\ve,{\bf p},1}+\phi_{\ve,{\bf p},2})^2\frac{1}{V(p_i)^2}(-\xi_{\sigma,i}^2R_1-\xi_{\sigma,i}^2R_2)\frac{\partial Pw_i}{\partial \tau(p_i)}\,dx\nonumber\\
 &&\,\,\,\,\,\,\,\,+O(\xi_\sigma)(\sum_{j\neq i}\xi_\sigma^2 |\nabla_{p_i} G_{\sqrt{D}}(\sigma p_i,\sigma p_j)|+\sum_{i=1}^k\xi_\sigma \ve^2 |h'(\ve p_i)|)\nonumber\\
 &&=-\xi_{\sigma,i}^2\int_{\R^2_+}w(y)^2R_2(y)\frac{\partial w}{\partial y_1}\,dy+O(\xi_\sigma^2\sigma \ve D\log\frac{\xi_\sigma}{\ve D})\nonumber\\
 &&=-\sum_{j\neq i}\xi_{\sigma}^2\frac{\partial G_{\sqrt{D}}(\sigma p_i,\sigma p_j)}{\partial \tau(p_i)}(\int_{\R^2_+}w^2y_1\frac{\partial w}{\partial y_1}\,dy\int_{\R^2_+}w^2\,dy)+O(\xi_\sigma^2\sigma \ve D\log\frac{\xi_\sigma}{\ve D})\nonumber\\
 &&=\nu_2\sum_{j\neq i}\xi_{\sigma}^2\frac{\partial G_{\sqrt{D}}(\sigma p_i,\sigma p_j)}{\partial \tau(p_i)}+O(\xi_\sigma^2\sigma \ve D\log\frac{\xi_\sigma}{\ve D})\nonumber\\
 &&= \nu_2\sum_{j\neq i}\xi_{\sigma}^2\frac{\partial G_0(\sigma p_i,\sigma p_j)}{\partial \tau(p_i)}+O(\xi_\sigma^2\sigma \ve D\log\frac{\xi_\sigma}{\ve D}),
 \end{eqnarray}
 where we have used $\sqrt{D}\log\frac{\xi_\sigma}{\ep D}\ll 1$ and
 \begin{equation}
 \label{nu2}
 \nu_2=
 \frac{1}{3}\int_{\R^2_+}w^3\,dy\int_{\R^2_+}w^2\,dy>0.
 \end{equation}

 Thus by (\ref{e:19}) and (\ref{e:20a}), we have
 \begin{eqnarray}\label{e:21}
 I_1&&=\nu_2\sum_{j\neq i}\xi_{\sigma}^2\frac{\partial G_{0}(\sigma p_i,\sigma p_j)}{\partial \tau(p_i)}-\ve^2\xi_\sigma \nu_1 \frac{\partial }{\partial \tau(\ve p_i)}h(\ve p_i)\nonumber\\
 &&+ O(\xi_\sigma^2\sigma \ve D\log\frac{\xi_\sigma}{\ve D}).
 \end{eqnarray}

 Next we estimate $I_2$:
 \begin{eqnarray*}
 I_2&&=\int_{\Omega_\ve}(\frac{(U+\phi)^2}{V+\psi}-\frac{(U+\phi)^2}{V})\frac{\partial Pw_i}{\partial \tau(p_i)}\,dx\\
 &&=-\int_{\Omega_\ve}\frac{(\xi_{\sigma,i}Pw_i+\phi)^2}{V^2}\psi\frac{\partial Pw_i}{\partial \tau(p_i)}\,dx\\
 &&+O(\xi_\sigma)(\sum_{j\neq i}\xi_\sigma^2 |\nabla_{p_i} G_{\sqrt{D}}(\sigma p_i,\sigma p_j)|+\sum_{i=1}^k\xi_\sigma \ve^2 |h'(\ve p_i)|)\\
 &&=-\frac{1}{3}\int_{\Omega_\ve}\frac{\partial (Pw_i)^3}{\partial \tau(p_i)}(\psi(x)-\psi(p_i))\,dx+ O(\xi_\sigma^2\sigma \ve D\log\frac{\xi_\sigma}{\ve D}).
 \end{eqnarray*}

 By the equation for $\psi$, we have
 \begin{equation*}
 \Delta \psi-\sigma^2\psi+2U\phi+\phi^2=0
 \end{equation*}
 and therefore
 \begin{eqnarray*}
 \psi(p_i+z)-\psi(p_i)&&=\int_{\Omega_\ve}[G_{\sqrt{D}}(\sigma p_i,\sigma y)-G_{\sqrt{D}}(\sigma p_i+\sigma z,\sigma y)](2U\phi+\phi^2)(y)\,dy\\
 &&=O(\xi_\sigma \sum_{j\neq i}\xi_\sigma^2|\nabla_{p_i} G_{\sqrt{D}}(\sigma p_i,\sigma p_j)||z|)+O(\xi_\sigma^3)R(z),
 \end{eqnarray*}
 where $R(z)$ is an even function in $z_1$.

 Thus we have
 \begin{eqnarray}\label{e:22}
 I_2=O(\xi_\sigma^2\sigma \ve D\log\frac{\xi_\sigma}{\ve D}).
 \end{eqnarray}

Combining the estimates for $I_1$ and $I_2$, (\ref{e:21}) and (\ref{e:22}), we have
 \begin{eqnarray*}
 W_{\ve,i}&:=&\int_{\Omega_\ve}S_1[U+\phi,V+\psi]\frac{\partial Pw_i}{\partial \tau(p_1)}\,dx\\
 &=&\xi_\sigma[\nu_2\sum_{j\neq i}\xi_{\sigma}\sigma G_{0}'(\sigma|p_i-p_j|)\frac{p_i-p_j}{|p_i-p_j|}-\ve^3 \nu_1 \frac{\partial^2 }{\partial \tau^2}h(\ve p^0)(p_i-p^0)]\nonumber\\
 &+& o(\xi_\sigma \sigma \ve D\log\frac{\xi_\sigma}{\ve D}).
  \end{eqnarray*}

 Thus $W_{\ve,i}=0$ is reduced to the following system:
 \begin{equation}
 \left\{\begin{array}{l}
-\nu_2\xi_{\sigma}\sigma G_{0}'(\sigma|p_1-p_2|)-\nu_1 \ve^3\frac{\partial^2 }{\partial \tau^2}h(\ve p^0)(p_1-p^0)=o(\sigma \ve D\log\frac{\xi_\sigma}{\ve D}),\\
\\
\nu_2\xi_{\sigma}\sigma (G_{0}'(\sigma|p_1-p_2|)- G_{0}'(\sigma|p_2-p_3|))-\nu_1 \ve^3\frac{\partial^2 }{\partial \tau^2}h(\ve p^0)(p_2-p^0)=o(\sigma \ve D\log\frac{\xi_\sigma}{\ve D}),\\
\\
\cdots  \\
\\
\nu_2\xi_{\sigma}\sigma (G_{0}'(\sigma|p_{k-1}-p_{k-2}|)-G_{0}'(\sigma|p_{k-1}-p_k|))-\nu_1 \ve^3\frac{\partial^2 }{\partial \tau^2}h(\ve p^0)(p_{k-1}-p^0)=o(\sigma \ve D\log\frac{\xi_\sigma}{\ve D}),\\
\\
\nu_2\xi_{\sigma}\sigma G_{0}'(\sigma|p_k-p_{k-1}|)-\nu_1 \ve^3\frac{\partial^2 }{\partial \tau^2}h(\ve p^0)(p_k-p^0)=o(\sigma \ve D\log\frac{\xi_\sigma}{\ve D}).
\end{array}
\right.
 \end{equation}

 We first solve the limiting case:
  \begin{equation}
 \left\{\begin{array}{l}
-\nu_2\xi_{\sigma}\sigma G_{0}'(\sigma|p^0_1-p^0_2|)-\nu_1 \ve^3\frac{\partial^2 }{\partial \tau^2}h(\ve p^0)(p^0_1-p^0)=0,\\
\\
\nu_2\xi_{\sigma}\sigma (G_{0}'(\sigma|p^0_1-p^0_2|)- G_{0}'(\sigma|p^0_2-p^0_3|))-\nu_1 \ve^3\frac{\partial^2 }{\partial \tau^2}h(\ve p^0)(p^0_2-p^0)=0,\\
\\
\cdots  \\
\\
\nu_2\xi_{\sigma}\sigma (G_{0}'(\sigma|p^0_{k-2}-p^0_{k-1}|)
-G_{0}'(\sigma|p^0_{k-1}-p^0_k|))-\nu_1 \ve^3\frac{\partial^2 }{\partial \tau^2}h(\ve p^0)(p^0_{k-1}-p^0)=0,\\
\\
\nu_2\xi_{\sigma}\sigma G_{0}'(\sigma|p^0_{k-1}-p^0_{k}|)-\nu_1 \ve^3\frac{\partial^2 }{\partial \tau^2}h(\ve p^0)(p^0_k-p^0)=0.
\end{array}
\right.
\label{above}
 \end{equation}

This system  is uniquely solvable with
\begin{equation}\label{e:32}
\sum_{i=1}^k p_i^0=p^0.
\end{equation}

Moreover, we have
\begin{eqnarray}\label{e:33}
\sigma(p^0_i-p^0_{i-1})&&=\log\frac{\xi_\sigma}{\ve D}-\frac{3}{2}\log\log\frac{\xi_\sigma}{\ve D}\\
&&-\log(-\frac{h''(p^0)\nu_1}{2 \nu_2})-\log[(i-1)(k+1-i)]+O(\frac{\log\log\frac{\xi_\sigma}{\ve D}}{\log \frac{\xi_\sigma}{\ve D}}),
\nonumber
\end{eqnarray}
where we have used the notation $h''(p^0)=\frac{\partial^2}{\partial \tau^2}h(P^0)<0$.

 From (\ref{e:33}), we know
 \begin{equation}
 \sum_{i=1}^jp_i^0=O(\frac{1}{\sigma}\log\frac{\ve D}{\xi_\sigma}) \mbox{ for }j=1,\cdots,k-1.
 \end{equation}

To find $p_i$ such that $W_{\ve,i}=0$, we expand $p_i=p_i^0+\tilde{p_i}$.
 Then adding
 the first $i$ equations, we have
\begin{equation}
-\nu_2\xi_{\sigma}\sigma G''_{0}(\sigma|p_i^0-p_{i+1}^0|)[\sigma(\tilde{p}_i-\tilde{p}_{i+1})+O(\sigma^2|\tilde{p}|^2)]
+\nu_1\sum_{j=1}^i\ve^3h''(p^0)\tilde{p}_i=o(\sigma \ve D\log\frac{\xi_\sigma}{\ve D}).
\end{equation}
 One can get that
 \begin{equation*}
 \tilde{p}_i=\tilde{p}_1(1+o(1))+O(\frac{1}{\sigma}), \mbox{ for }i=2,\cdots,k.
 \end{equation*}
By the last equation
 \begin{equation}
 \sum_{i=1}^k\nu_1\ve^3h''(p^0)\tilde{p}_i=o(\sigma \ve D\log\frac{\xi_\sigma}{\ve D}).
 \end{equation}

Thus
\begin{equation}
\nu_1 \ve^3h''(p^0)k\tilde{p}_1(1+o(1))=o(\sigma \ve D\log\frac{\xi_\sigma}{\ve D})+O(\frac{\ve^3}{\sigma}),
\end{equation}
 From the equation above, we can estimate $\tilde{p}_1$ by
 \begin{equation}
 \tilde{p}_1=o(\frac{1}{\sigma}\log \frac{\ve D}{\xi_\sigma}).
 \end{equation}
 In conclusion, we solve $W_{\ve,i}=0$ with
 \begin{equation*}
 \tilde{p}_i=o(\frac{1}{\sigma}\log \frac{\ve D}{\xi_\sigma}).
 \end{equation*}

 Thus we have proved the following proposition:
 \begin{proposition}\label{pro1a}
 For $\max\{\sigma, D\}$ small enough, there exists ${\bf p}^\ve\in \Lambda_k$ with $P_i\to P^0$ such that $W_{\ve,i}=0$.
 \end{proposition}

 Finally, we complete the proof of Theorem \ref{existenceth}.

 {\em Proof}. By Proposition \ref{pro1a}, there exists ${\bf P}^\ve\to {\bf P}^0$, such that $W_{\ve}({\bf p}^\ve)=0$. In other words, we have
  \begin{equation}
  S_\ve\vect{U+\phi}{V+\psi}=0.
 \end{equation}
 Moreover, by the maximum principle, $(U,V)>0$ and the solution satisfies all the properties of Theorem \ref{existenceth}.

 \qed

\section{Study of the large eigenvalues}
\setcounter{equation}{0}

We consider the stability of the steady-state $(u,v)$ constructed in Theorem 2.1.

In this section,
we first study the large eigenvalues which satisfy $\lambda_\ve\to \lambda_0\neq 0$ as $\max\{\sigma,D\}\to 0$.

Linearizing the system around the equilibrium states $(u,v)$  obtained in Theorem \ref{existenceth}, we obtain the following eigenvalue problem:
\begin{equation}
\label{equ1}
\left\{\begin{array}{l}
\Delta \phi-\phi+\frac{2u}{v}\phi-\frac{u^2}{v^2}\psi=\lambda \phi,\\[2mm]
\Delta \psi-\sigma^2\psi+2u\phi=\tau \lambda \sigma^2\psi
\end{array}
\right.
\end{equation}
for $(\phi,\psi)\in H^2_N(\Omega_\ve)\times H^2_N(\Omega_\ve)$.

In this section, since we study the large eigenvalues, we may assume that $|\lambda_\ve|\geq c>0$ for $\max\{\sigma, D\} $ small enough. If $\mbox{Re}(\lambda_\ve)\leq -c<0$, then $\lambda_\ve$ is a stable large eigenvalue, we are done. Therefore, we may assume that $\mbox{Re}(\lambda_\ve)\geq -c$ and for a subsequence $\max\{\sigma,D\}\to 0$, $\lambda_\ve\to \lambda_0\neq 0$. We shall derive the limiting eigenvalue problem which is given by a coupled system of NLEPs.

The second equation of (\ref{equ1}) is equivalent to
\begin{equation}
\Delta \psi-\sigma^2(1+\tau \lambda_\ve)\psi+2u\phi=0.
\end{equation}

We introduce the following notation:
\begin{equation*}
\sigma_\lambda=\sigma \sqrt{1+\tau \lambda_\ve},
\end{equation*}
where in $\sqrt{1+\tau \lambda_\ve}$, we take the principal part of the square root.

Let us assume that
\begin{equation*}
\|\phi\|_{H^2(\Omega_\ve)}=1.
\end{equation*}
We cut off $\phi$ as follows:
\begin{equation}
\phi_{\ve,j}=\phi_{\ve}\chi_\ep(z-p_j), \ j=1,\cdots,k,
\end{equation}
where the cutoff function $\chi_\ep$ has been defined in (\ref{chive}).

From (\ref{equ1}) and the exponential decay of $w$, it follows that
\begin{equation}
\phi_{\ve}=\sum_{j=1}^k\phi_{\ve,j}+O(\ve^5).
\end{equation}

Then by a standard procedure (see \cite{gt}, Section 7.12), we extend $\phi_{\ve,j}$ to a function defined on $\R^2$ such that
\begin{equation*}
\|\phi_{\ve,j}\|_{H^2(\R^2)}\leq C\|\phi_{\ve,j}\|_{H^2(\Omega_\ve)}, \ j=1,\cdots,k.
\end{equation*}

Since $\|\phi_\ve\|_{H^2(\Omega_\ve)}=1$, $\|\phi_{\ve,j}\|_{H^2(\R^2)}\leq C$ . By taking a subsequence, we may assume that $\phi_{\ve,j}\to \phi_j$ as $\max\{\sigma, D\}\to 0$ in $H^1(\R^2)$ for some $\phi_j\in H^1(\R^2)$ for $j=1,\cdots,k$.

By (\ref{equ1}), we have
\begin{eqnarray*}
\psi_{\ve}(p_j)&&=\int_{\Omega_\ve}G_{\sqrt{D}}(\sigma_\lambda p_j,\sigma_\lambda x)2u\phi_{\ve}(x)\,dx\\
&&=\int_{\Omega_\ve}G_{\sqrt{D}}(\sigma_\lambda p_j,\sigma_\lambda x)
2(\sum_{j=1}^k\xi_{\sigma,j}Pw_j\phi_{\ve,j}+O(\xi_\sigma^2))\,dx\\
&&=\frac{1}{\pi}\log\frac{1}{\sigma_\lambda}\int_{\R^2}2\xi_{\sigma,i}w_i\phi_{\ve,i}(1+o(1))\,dx.
\end{eqnarray*}

Substituting the above equation into the first equation of  (\ref{equ1}), letting $\max\{\sigma,D\}\to 0$, and using the expansion of $\xi_{\sigma,j}$, we arrive at the following nonlocal eigenvalue problem (NLEP):
\begin{equation}\label{e:20}
\Delta \phi_j-\phi_j+2w\phi_j-\frac{2}{1+\tau\lambda}\,
\frac{\int_{\R^2_+}w\phi_j\,dx}{\int_{\R^2_+}w^2\,dx}w^2=\lambda_0\phi_j, \ j=1,\cdots,k.
\end{equation}
By Theorem 3.5 in \cite{ww-book}, (\ref{e:20}) has only stable eigenvalues if $\tau$ is small enough.

In conclusion, we have shown that the large eigenvalues of the $k$-peaked solutions given in Theorem 2.1 are all stable if $\tau$ is small enough.

\section{Study of the small eigenvalues}
\setcounter{equation}{0}

Now we study the eigenvalue problem (\ref{equ1}) with respect to small eigenvalues. Namely, we assume that $\lambda_\ve \to 0$ as $\max\{\sigma,D\}\to 0$. We will show that the small eigenvalues in leading order are related to the matrix $\mathcal{M}({\bf p}^0)$ given in (\ref{mp0}) which is computed from the Green's function. Our main result in this section says that if $\lambda_\ve\to 0$, then in leading order
\begin{equation}
\lambda_\ve\sim \xi_\sigma \sigma_0,
\end{equation}
where $\sigma_0$ is an eigenvalue of $\mathcal{M}({\bf p}^0)$.
We will show that all the eigenvalues of $\mathcal{M}({\bf p}^0)$ have negative real part provided that the eigenvector is orthogonal to
$(1,1,...,1)^T$.

However, for the eigenvector $(1,1,...,1)^T$ the eigenvalue of $\mathcal{M}({\bf p}^0)$ is zero, the leading order term in the eigenvalue expansion vanishes and the next order term is needed to prove stability. To establish it we have to compute the contribution from the boundary curvature.
It follows that for a local maximum point of the boundary curvature this eigenvalue has negative real part. Whereas the Green's function part is of order
$\ep^3 \log\frac{\xi_\sigma}{\ep D}$,
the part from the boundary curvature is of order $\ep^3$.
Thus the small eigenvalues of (\ref{equ1}) are all stable.

To compute the small eigenvalues, we need to expand the spike cluster solution to higher order. Then we expand the eigenfunction and compute the small eigenvalues. This will be done in Appendix B.
The key estimates are given in Lemma \ref{lemma4}.

We compute the small eigenvalues using Lemma \ref{lemma4}. Comparing l.h.s and r.h.s, we obtain
\begin{eqnarray}
\label{eig11}
-\nu_2\xi_{\sigma}^2\mathcal{M}({\bf p}^0){\bf a}_\ve(1+o(1))=\lambda_\ve\xi_\sigma {\bf a}_\ve\int_{\R^2_+}(\frac{\partial w}{\partial y_1})^2\,dy(1+o(1)),
\end{eqnarray}
where $\nu_2$ has been defined in (\ref{nu2}).
Further, we have
\begin{eqnarray}
\label{mp0}
\mathcal{M}({\bf p}^0)=(m_{ij}({\bf p}^0))_{i,j=1}^k,
\end{eqnarray}
where
\begin{eqnarray*}
m_{ij}({\bf p})&&=\Big[[\nabla^2_{\tau(p_i)}G_{0}(\sigma p_i,\sigma p_{i-1})+\nabla^2_{\tau(p_i)}G_{0}(\sigma p_i,\sigma p_{i+1})]\delta_{ij}\\
&&
-[\delta_{i,j+1}\nabla^2_{\tau(p_i)}G_{0}(\sigma p_i,\sigma p_{i-1})+\delta_{i,j-1}\nabla^2_{\tau(p_i)}G_{0}(\sigma p_i,\sigma p_{i+1})]\Big].
\end{eqnarray*}

Using the estimates for $p^0_i$ in (\ref{e:32}) and (\ref{e:33}), we have
\begin{eqnarray*}
m_{ij}({\bf p}^0)&&=-\frac{h''(p^0)\nu_1}{2\nu_2}\frac{\ve D}{\xi_\sigma}\log\frac{\xi_\sigma}{\ve D}\sigma^2\Big[-(i-1)(k+1-i)\delta_{j,i-1}-i(k-i)\delta_{j,i+1}\\
&&\ \ \ \ \ \ \ \ \ \ \ \ \ \ \ \ \ \ \ \ +((i-1)(k+1-i)+i(k-i))\delta_{ij}\Big].
\end{eqnarray*}
This shows that if all the eigenvalues of $\mathcal{M}({\bf p}^0)$ have positive real part, then the small eigenvalues are stable. On the other hand, if $\mathcal{M}({\bf p}^0)$ has eigenvalues with negative real part, then there are eigenfunctions and eigenvalues to make the system unstable.
Next we study the spectrum of the $k\times k$ matrix $A$ defined by
\begin{eqnarray*}
&&a_{s,s}=(s-1)(k-s+1)+s(k-s), \quad s=1,\cdots,k,\\
&&a_{s,s+1}=a_{s+1,s}=-s(k-s), \quad s=1,\cdots,k,\\
&&a_{s,l}=0, \quad |s-l|>1.
\end{eqnarray*}
We have the following result from Lemma 16 in \cite{WW8}:
\begin{lemma}
\label{l71}
The eigenvalues of the matrix $A$ are given by
\begin{equation}
\lambda_n=n(n+1), \quad n=0,\cdots,k-1.
\end{equation}
\end{lemma}

By Lemma \ref{l71}, the eigenvalues of $\mathcal{M}({\bf p}^0)$ are all positive except for a single eigenvalue zero with eigenvector $(1,1,\ldots,1)^T$.

Equation (\ref{eig11}) shows that the small eigenvalues $\lambda_\ve$ are
\begin{equation}
\lambda_\ve\sim -\frac{\nu_2\xi_\sigma}{\int_{\R^2_+}(\frac{\partial w}{\partial y_1})^2\,dy}\sigma(\mathcal{M}({\bf p}^0)).
\end{equation}

We remark that the scaling of these small eigenvalues is
\[
\lambda_\ep \sim c_5 \ep^3 \log\frac{\xi_\sigma}{\ep D}
\]
for some $c_5<0$.

However, one of the eigenvalues of $\mathcal{M}({\bf p}^0)$
is exactly zero, with eigenvector $(1,1,\ldots,1)^T$. To determine the sign of the real part of this eigenvalue, we have to expand to the next order. By considering contributions for the curvature of the boundary $\partial\Om$ we have computed in Appendix B that for this eigenvalue
\[
\lambda_\ep \sim c_6 \ep^3
\]
for some $c_6<0$.

To summarise, there are small eigenvalues of two orders which differ by the logarithmic factor $\log\frac{\xi_\sigma}{\ep D}$.

 \section{Appendix A: Linear Theory}
 \setcounter{equation}{0}

In this section we prove Proposition \ref{linboun}.
 We follow the Liapunov-Schmidt reduction method. Suppose that to the contrary, there
 exist sequences $\ve_n, \sigma_n, {\bf p}_n$ and $\Sigma_n$ with $\ve_n,\sigma_n\to 0$ and $\Sigma_n=\vect{\phi_{n}}{\psi_n}\in \mathcal{K}_{\ve_n,{\bf p}_n}^\perp$ such that
 \begin{equation}
 \left\{\begin{array}{l}
 \Delta \phi_n-\phi_n+\frac{2U}{V}\phi_n-\frac{U^2}{V^2}\psi_n=f_n,\\[2mm]
 \Delta \psi_n-\sigma^2\psi_n+2U\phi_n=g_n,
 \end{array}
 \right.
 \end{equation}
 where
 \begin{equation}
 \|\pi_{\ve,p}\circ f_n\|_{L^2(\Omega_\ve)}\to 0,
 \end{equation}
 \begin{equation}
 \|\xi_{\sigma}^{-1}g_n\|_{L^2(\Omega_\ve)}\to 0
 \end{equation}
 and
 \begin{equation}
 \|\phi_n\|_{H^2(\Omega_\ve)}+\|\psi_n\|_{H^2(\Omega_\ve)}=1.
 \end{equation}
 We now show that this is impossible. To simplify notation, we omit the index $n$. In the first step we show that the linearised problem given above tends to a limit problem as $\max\{\sigma, D\}\to 0$.

 We define
 \begin{equation}
 \phi_{\ve,i}=\phi(x)\chi(x-p_i) \mbox{ for }i=1,\cdots,k,
 \end{equation}
 and
 \begin{equation}
 \phi_{\ve,k+1}=\phi_{\ve}-\sum_{i=1}^k\phi_{\ve,i}.
 \end{equation}
 It is easy to see that $\phi_{\ve,k+1}=o(1)$ in $H^2(\Omega_\ve)$, since it satisfies
 \begin{equation*}
 \Delta \phi_{\ve,k+1}-\phi_{\ve,k+1}=o(1) \mbox{ in }H^2(\Omega_\ve).
 \end{equation*}
 We define $\psi_{\ve,i}$ by
 \begin{equation}
 \left\{\begin{array}{l}
 \Delta \psi_{\ve,i}-\sigma^2\psi_{\ve,i}+2U\phi_{\ve,i}=0, \mbox{ in }\Omega_\ve,\\
 \frac{\partial \psi_{\ve,i}}{\partial \nu}=0 \mbox{ on }\partial \Omega_\ve.
 \end{array}
 \right.
 \end{equation}
 Note that since $\xi_{\sigma}^{-1}\|g_n\|_{L^2(\Omega_\ve)}\to 0$, we also have $\|\psi_{\ve,k+1}\|_{H^2(\Omega_\ve)}=o(1)$.

Next by the equation satisfied by $\psi_{\ve}$, we have
\begin{equation}
\psi_{\ve,n}(x)=\int_{\Omega_\ve}G_{\sqrt{D}}(\sigma x,\sigma y)(2U\phi_{\ve}-g_n)(y)\,dy.
\end{equation}
 So at $x=p_i$, we calculate
 \begin{eqnarray*}
 \psi_{\ve}(p_i)&=&\int_{\Omega_\ve}G_{\sqrt{D}}(\sigma p_i,\sigma y)(2U\phi_{\ve}-g_n)(y)\,dy\\
 &=&\int_{\Omega_\ve}[\frac{1}{\pi}\log\frac{1}{\sigma|p_i-y|}+\tilde{H}(\sigma p_i,\sigma y)]2U\phi_{\ve}(y)\,dy\\
 &+&\int_{\Omega_\ve}G_{\sqrt{D}}(\sigma p_i,\sigma y)g_n(y)\,dy\\
 &=&2\int_{\R^2_+}w(y)\phi_{\ve,i}\,dy(1+o(1))+O(\log\frac{1}{\sigma}\|g_n\|_{L^2(\Omega_\ve)})\\
 &=&2\int_{\R^2_+}w(y)\phi_{\ve,i}\,dy+o(1).
 \end{eqnarray*}

Substituting the above equation into the first equation of (\ref{equ1}), letting $\max\{\sigma, D\}\to 0$, we can show that
\begin{equation}
\phi_{\ve,i}\to \phi_i \mbox{ in }H^2(\R^2_+),
\end{equation}
and
\begin{equation}
\phi_i\in \{\phi\in H^2(\R^2_+)|\int_{\R^2_+}\phi\frac{\partial w}{\partial y_1}\,dy=0\}:=K_0^\perp,
\end{equation}
where $\phi_i$ is solution of the following nonlocal problem:
\begin{equation}
\Delta \phi_i-\phi_i+2w\phi_i-\frac{2\int_{\R^2_+}w\phi_i\,dy}{\int_{\R^2_+}w^2\,dy}w^2(y)\in C_0^\perp,
\end{equation}
 where $C_0^\perp, K_0^\perp$ denote the orthogonal complements with respect to the scalar product of $L^2(\R^2_+)$ in the space of $H^2(\R^2_+)$ and $L^2(\R^2_+)$ respectively.

By Theorem 1.4 in \cite{W4}, we know that $\phi_i=0 ,i=1,\cdots,k$.

By taking the limit in the equation satisfied by $\psi_{\ve}$, we see that this implies that $\psi_{\ve}\to 0$ in $H^2(\Omega_\ve)$. This contradicts the assumption
\begin{equation}
 \|\phi_n\|_{H^2(\Omega_\ve)}+\|\psi_n\|_{H^2(\Omega_\ve)}=1.
 \end{equation}
 This proves the boundedness of the linear operator $\mathcal{L}_{\ep,{\bf p}}$.

To complete the proof of Proposition \ref{linboun}, we just need to show that conjugate operator to $\mathcal{L}_{\ep,{\bf p}}$ (denoted by $\mathcal{L}_{\ep,{\bf p}}^*$) is injective from $\mathcal{K}_{\ep,{\bf }}^{\perp}$ to $\mathcal{C}_{\ep,{\bf p}}^{\perp}$. The proof for $\mathcal{L}_{\ep,{\bf p}}^*$ follows almost the same process as for $\mathcal{L}_{\ep,{\bf p}}$ and therefore it is omitted.

The proof is complete.

\section{Appendix B: Computation of the small eigenvalues}

In this appendix we will compute the small eigenvalues. First we expand the solution to a higher degree of accuracy than in Section 3. Then we expand the eigenfunctions and finally we calculate the small eigenvalues.

\subsection{Further expansion of the solution}

In this subsection, we further improve our expansion to the solutions derived in Section 3.

First we define
\begin{equation*}
u_i(x)=u_\ep(x)\chi_\ve(x-p_i), \ i=1,\cdots,k,
\end{equation*}
where $u_\ep$ is the exact boundary cluster solution derived in Section 3-5 and $\chi_\ve $ is the cutoff function given in (\ref{chive}).
It is easy to see that
\begin{equation*}
u(x)=\sum_{i=1}^ku_i(x)+O(\ve^{5}),
\end{equation*}
We will derive an approximation to $u_i$ which is more accurate than that given in Section 3.

Our main idea is to start with a single boundary spike solution of the Gierer-Meinhardt system in a disk $B_R$ of radius $R$ such that the curvature at the centres of the boundary spikes on the disk and the domain $\partial \Om$ agree.
Further, the boundary spike solution in the disk is invariant under rotations of the solution which results in a zero eigenvalue.
Then the small eigenvalue of the boundary spike solution in $\Om$ can be computed by a perturbation analysis of the disk case.

The single boundary spike solution $(u_0,v_0)$
in the ball $B_R$ with $\max_{B_R}{u_0}=u_0(0,R)$ in polar coordinates solves the following system:
\begin{equation}
\left\{\begin{array}{l}
\ep^2 \Delta u_0-u_0+\frac{u_0^2}{T_R[u_0^2]}=0
\quad \mbox{ in }B_R,
\\[2mm]
\frac{\partial u_0}{\partial r}=0
\quad \mbox{ on }\partial B_R,
\label{u0def}
\end{array}\right.
\end{equation}
where
$T_R[u_0^2]=v_0$ is the solution of the inhibitor equation
\begin{equation}
\left\{\begin{array}{l}
D\Delta v_0-v_0+u_0^2=0\quad \mbox{ in }B_R,
\\[2mm]
\frac{\partial v_0}{\partial r}=0
\quad \mbox{ on }\partial B_R.
\end{array}\right.
\label{v0def}
\end{equation}
It can be constructed from the ground state $w$ following the approach in Section 3
by using a fixed-point argument in the space of even functions around $\alpha=0$ and no Liapunov-Schmidt reduction is needed.

Note that the single-boundary spike solution $u_0$ is invariant under rotations of the solution.
Therefore we can apply $\frac{\partial}{\partial\alpha}$ in
(\ref{u0def}) and get
\begin{equation}
\left\{\begin{array}{l}
\ep^2 \Delta \frac{\partial u_0}{\partial \alpha}-\frac{\partial u_0}{\partial \alpha}+\frac{2u_0\frac{\partial u_0}{\partial \alpha}}{T_R[u_0^2]}- \frac{u_0^2}{(T_R[u_0^2])^2}
\frac{\partial}{\partial\alpha}T_R[u_0^2]=0
\quad \mbox{ in }B_R,
\\[2mm]
\frac{\partial^2 u_0}{\partial \alpha\partial r}=0
\quad \mbox{ on }\partial B_R.
\end{array}\right.
\label{u0der}
\end{equation}

As a preparation for this perturbation analysis, we represent the boundary $\partial\Om$ in a neighbourhood of the centre of the spike $p_i$ in polar coordinates and deform it to a circle with the same curvature.
This will imply that the perturbation of the boundary will only be in order $\ep^3$.

Near the point $p_i$ we have expanded the boundary $\partial\Om$ in Section 3.1 using Cartesian coordinates. We have derived that
$\rho(0)=\rho'(0)=0$ and $\rho''(0)$ is the curvature of $\partial\Om$ at $p_i$.
Recall that $\rho$ was used to flatten the boundary $\partial\Om$ near $p_i$ to a line.

 Using polar coordinates $(\phi,r)$ such that $x_1-p_1=r\sin\phi$, $x_2-p_2=R-r\sin\phi$,
the boundary $\partial\Om$ can be represented as
\[
r=f(\phi).
\]
The radius $R$ will be chosen such that
$\rho''(0)=\frac{1}{R}.$
We now derive the function $f(\phi)$ from the  function $\rho(x_1-p_1)$ introduced in Section 3.1.

Substituting the Taylor expansion of $\rho$ given in equation (\ref{taylor})
into $f(\phi)^2$ gives
\[
f(\phi)^2=(x_1-p_1)^2+(x_2-p_2)^2
=r^2\sin^2\phi+(R-\rho(r\sin\phi))^2
\]
\[
=R^2+r^2\sin^2\phi
-2R\rho(r\sin\phi)+\rho^2(r\sin\phi).
\]
We expand
\[
f(\phi)=R+\alpha\phi+\beta\phi^2
+\gamma\phi^3+\delta\phi^4,
\]
and compute the coefficients $\alpha,\,\beta,\,\gamma,\,\delta$ by matching powers $\phi^i$ for $i=1,2,3,4$.
First we get
\[\alpha=0.\]
Second we have
\[ 2\beta=R(1-\rho''(0)R) \]
which implies that
\[ \beta=0 \mbox{ provided } \rho''(0)=\frac{1}{R} \]
and from now on we choose $R$ such that this condition is satisfied.
Third we compute
\[ 2R\gamma=-\frac{1}{3} \rho^{(3)}(0)R^4 \]
which gives
\[
\gamma=-\frac{1}{6} \rho^{(3)}(0)R^3.
\]
Finally, we get
\[
2R\delta=-\frac{1}{3}R^2+\frac{1}{3}R^2
-\frac{1}{12}\rho^{(4)}(0)R^5
+\frac{1}{4}\left(\frac{1}{R}\right)^2 R^4
\]
which gives
\[
\delta=-\frac{1}{24}\rho^{(4)}(0)R^4
+\frac{1}{8} R.
\]
To summarise, we have
\[
f(\phi)=R+\frac{1}{6}f^{(3)}(0)\phi^3+\frac{1}{24}f^{(4)}(0)\phi^4+O(\phi^5),
\]
where $f(\phi)$ denotes the radius and $\phi$ the angle and
\[
f^{(3)}(0)=- \rho^{(3)}(0)R^3,
\]\[
f^{(4)}(0)=-\rho^{(4)}(0)R^4+3R.
\]
The point $\phi=0$ with $f(0)=R$
and $f'(0)=f''(0)=0$ corresponds to $p_i$.
Note that $f(0)=R$ and $f'(0)=0$ follow from $\rho(0)=\rho'(0)=0$.

Further, we have $f''(0)=0$ due to the choice of the leading term $f(0)=R$. Since we have $f(\phi)=r$, we get $r=R+O(\phi^3)$ which enables us to replace $r$ by $R$ in the derivation of $f(\phi)$.

In polar coordinates, for $r=R$ we get a point on $\partial B_R$.
We change variables such that for the variables $(\phi,r')$ at $r'=R$ we get a point on $\partial\Om$. Thus $r'=R$ has to map into $r=f(\phi)$. This is achieved by defining $r'=r-f(\phi)+R$ for each $\phi$.

Then the Laplacian is transformed as follows:
\[
\Delta=\Delta' +\frac{1}{r^2}(f'(\phi)^2)\partial_{\phi\phi} -\frac{2}{r^2}f'(\phi)\partial_{r\phi}
-\frac{1}{r^2} f''(\phi)\partial_{\phi},
\]
using partial derivatives $\partial_r=\frac{\partial}{\partial r}$ etc. and
\[
\Delta=
\frac{1}{r}\partial_r\left(r\partial_r\right)
+\frac{1}{r^2}\partial_{\phi\phi},
\]
\[
\Delta'=
\frac{1}{r'}\partial_{r'}\left(r'\partial_{r'}
\right)+\frac{1}{r'^2}\partial_{\phi\phi}
\]
for the transformed variables $(\phi,r')$.

We introduce rescaled variables $(\alpha,b)$ inside the spike such that
\[\ep\alpha=\phi,\quad \ep b=R-r'.\]
Then in rescaled variables $(\alpha,b)$ we have
\[
g(\ep\alpha)=R-f(\ep\alpha)=
\frac{1}{6}\rho^{(3)}(0)R^3\ep^3\alpha^3
+ \frac{1}{24}\rho^{(4)}(0)R^4\ep^4\alpha^4
-\frac{1}{8}R\ep^4\alpha^4+O(\ep^5\alpha^5).
\]
This implies that in rescaled variables we get
\[
\Delta=\Delta' +\frac{1}{r^2}(g'(\ep\alpha)^2)\partial_{\alpha\alpha} -
\frac{2}{r^2}g'(\ep\alpha)\partial_{b\alpha}
+\frac{\ep}{r^2} g''(\ep\alpha)\partial_{\alpha}.
\]
Using
\[
g'(\ep\alpha)=\frac{1}{2}g^{(3)}(0)\ep^2\alpha^2
+\frac{1}{6}g^{(4)}(0)\ep^3\alpha^3
+O(\ep^4\alpha^4)
\]
and
\[
g''(\ep\alpha)= g^{(3)}(0)\ep\alpha
+\frac{1}{2}g^{(4)}(0)\ep^2\alpha^2
+O(\ep^3\alpha^3),
\]
we have
\[
g'(\ep\alpha)=O(\ep^2\alpha^2)
\]
and
\[
g''(\ep\alpha)= O(\ep\alpha).
\]
 Thus second, third and fourth terms in the Laplacian are small and they can be estimated as follows:
$O(\ep^4\alpha^4)$, $O(\ep^2\alpha^2)$ and $O(\ep^2 \alpha)$, respectively.


 Comparing with Cartesian coordinates $(z_1,z_2)$ for the $\ep$-scale inside the spike used in Section 3.1, by elementary trigonometry we get
 \[
 \ep (z_2-p_2)=R-(R-\ep b)\cos(\ep\alpha)
 \]\[
 \sim \ep b+\ep^2\frac{1}{2}R\alpha^2-\ep^3 \frac{1}{2}b\alpha^2+O(\ep^4)
 \]
 and
\[
\ep (z_1-p_1)=(R-\ep b)\sin(\ep\alpha)
\]\[
\sim \ep R\alpha
-\ep^2 \alpha b-\ep^3\frac{1}{6} R\alpha^3 +O(\ep^4).
\]

Now we approximate the exact solution for the activator $u_\ep$  as follows:
\begin{equation}
\label{uappr}
u_\ep=\bar{w}+\ep^3  \bar{v}^{(1)}+\ep^2 (\bar{v}^{(2)}+ \bar{v}^{(3)}
)+O(\ep^4).
\end{equation}
Near the centre $p_i$ of each spike
we have
\begin{equation}
\label{uapprloc}
u_i=\xi_{\sigma,i}\bar{w}_i+\ep^3 \xi_{\sigma,i} \bar{v}_i^{(1)}+\ep^2 \xi_{\sigma,i} (\bar{v}_i^{(2)}+ \bar{v}_i^{(3)}
)+O(\ep^4),
\end{equation}
where we have used the notation
\begin{equation}
 \bar{w}_i(x)=\xi_{\sigma,i}^{-1}\bar{w}(x-p_i)\chi_{\ep}(x-p_i),
 \quad
 \bar{v}_i^{(j)}(x)=\xi_{\sigma,i}^{-1}\bar{v}^{(j)}(x-p_i)
 \chi_{\ve}(x-p_i)
\end{equation}
for $i=1,2,\ldots,k,\, j=1,2,3$
and $\chi_\ve $ is the cutoff function defined in (\ref{chive}).

Let us derive the terms in this expansion step by step.
We start from the single boundary spike solution $u_0$ defined in (\ref{u0def}) in a ball of radius $R$ such that $\rho''(0)=\frac{1}{R}$.
Using polar coordinates, we can represent
$u_0(\ep\alpha, R-r)$ and
this function satisfies the Neumann boundary condition in a ball:
\[
\left.\frac{\partial u_0}{\partial r}\right|_{r=R}=0.
\]
However, the function $u_0$ does not satisfy the boundary condition at $\partial \Omega$ (and for $r=R$ we reach the boundary of the disk (circle) but not the domain boundary $\partial\Om$).
Recall that $r-f(\phi)=r'-R$. Thus, to get a function with a better approximation to the boundary condition at $\partial \Omega$ (and such that for $r=R$ we reach $\partial\Om$), we define
\[
\bar{w}(\epsilon\alpha,r')
=u_0(\epsilon\alpha,r).
\]
Then for $r'=R$, we have
\[
\bar{w}(\epsilon\alpha,R)=u_0(\epsilon\alpha,f(\epsilon\alpha)),
\]
i.e. for $r'=R$ the arguments of the function $\bar{w}$ are contained in $\partial\Om$ and the arguments of $u_0$ are contained in $\partial B_R$.
This means that the function $f$ deforms the boundary to a circle (in the same way as in Section 3.1 $\rho$ deforms the boundary to a straight line). Since the circle also takes into account the curvature it gives a better approximation to the boundary than the straight line and the approximate spike solution will give a better approximation to the exact solution than the approximations in Section 3.1. Note that the boundary spike solution in the ball is invariant under rotation (in the same way as the boundary spike is translation invariant in half space).

Now we calculate the radial derivative of $\bar{w}$ as follows:
\begin{equation}
\frac{\partial\bar{w}(\epsilon\alpha, r')}{\partial r'}|_{r'=R}
=\frac{\partial u_0(\epsilon\alpha,r)}{\partial r}|_{r=f(r)}=0.
\label{barwboun}
\end{equation}
The outward unit normal vector in polar coordinates is given by
\[
\nu=\frac{1}{\sqrt{f'(\phi)^2+f(\phi)^2}}
(-f'(\phi),f(\phi)).
\]
Using rescaled variables, this implies
\[
\nu=\frac{1}{\sqrt{g'(\ep\alpha)^2+(R-g(\ep\alpha))^2}}
(g'(\ep\alpha),R-g(\ep\alpha))
\]
\[
=\frac{1}{\sqrt{
\frac{1}{4}g^{(3)}(0)^2\ep^4\alpha^4
+(R-\frac{1}{6}g^{(3)}(0)\ep^3\alpha^3)^2
+O(\ep^4\alpha^4)
}}
\]\[\times
(\frac{1}{2}g^{(3)}(0)\ep^2\alpha^2
+\frac{1}{6}g^{(4)}(0)\ep^3\alpha^3
+O(\ep^4\alpha^4),
R-\frac{1}{6}f^{(3)}(0)\ep^3\alpha^3)
\]
Subtracting the radial unit vector $e_r=(0,1)$, this implies
\[
\nu-e_r=\frac{(\frac{1}{2}g^{(3)}(0)\ep^2\alpha^2
+\frac{1}{6}g^{(4)}(0)\ep^3\alpha^3,
O(\ep^3\alpha^3))}{R+O(\ep^3\alpha^3)}.
\]
Using that $\|\bar{w}_i-w\|_{H^2(\Om_\ep)}=O(\ep\alpha)$ and
that $\bar{w}$ satisfies (\ref{barwboun}),
the outward normal derivative of $\bar{w}$ is computed as
\begin{equation}
\left\{\begin{array}{l}
\frac{\partial \bar{w}_i}{\partial \nu}
=
\left(
\frac{\partial \bar{w}_i}{\partial \nu}
-
\frac{\partial \bar{w}_i}{\partial r}
\right)
+
\frac{\partial \bar{w}_i}{\partial r}
\\[2mm]
=\frac{1}{R}w'(|y|)
(\frac{1}{2}g^{(3)}(0)\ep^2\alpha^2
+\frac{1}{6}g^{(4)}(0)\ep^3\alpha^3
+O(\ep^4\alpha^4)
).
\end{array}\right.
\label{barwb}
\end{equation}

Now we compute the terms
$\ep^3\bar{v}^{(1)}$ (even around $\alpha=0$) and $\ep^2\bar{v}^{(2)}$ (odd around $\alpha=0$) in the expansion (\ref{uapprloc})
near $p_i$
such that the solution satisfies the Neumann boundary condition to higher order.

Let $\bar{v}_i^{(1)}$ satisfy
\[
\left\{\begin{array}{l}
\Delta v-v=0 \quad\mbox{ in } B_R,
\\[2mm]
\frac{\partial v}{\partial b}=
\frac{1}{6R}g^{(4)}(0)\alpha^3w'(|y|)
\quad \mbox{ on }\partial B_R.
\end{array}\right.
\]
Let $\bar{v}_i^{(2)}$ be given by
\[
\left\{\begin{array}{l}
\Delta v-v=0 \quad\mbox{ in } B_R,
\\[2mm]
\frac{\partial v}{\partial b}=
\frac{1}{2R}g^{(3)}(0)\alpha^2w'(|y|)
\quad
\mbox{ on }\partial B_R.
\end{array}\right.
\]
Substituting the expansion (\ref{uappr}) of the solution in the activator
equation, we get
\[
S_1(\bar{w}+\ep^3\bar{v}^{(1)}+\ep^2\bar{v}^{(2)}+O(\ep^4),
T[\bar{w}+\ep^3\bar{v}^{(1)}+\ep^2\bar{v}^{(2)}+O(\ep^4)])\\
\]\[
=\sum_{i=1}^k\xi_{\sigma,i}
((\bar{w}_i+\ep^3\bar{v}_i^{(1)}
+\ep^2\bar{v}_i^{(2)}
+O(\ep^4))^2-\bar{w}_i^2)
+\sum_{i=1}^k\xi_{\sigma,i}^2 \bar{w}_i^2(\frac{1}{V(x)}-\frac{1}{V(p_i)})
+O(\xi_\sigma\ve^4).
\]

We calculate for $x=p_i+z$
\begin{eqnarray*}
[(\bar{w}_i(x)+\ep^3\bar{v}^{(1)}(x)
+\ep^2\bar{v}^{(2)}(x)
+O(\ep^4)]^2-\bar{w}_i^2(x)
&=&2\bar{w}_i(z)(
\ve^3\bar{v}_i^{(1)}(z)
+\ve^2\bar{v}_i^{(2)}(z)
+O(\ve^4))\\
&:=&\ve^3  \bar{R}_{1,i}(z)+\ve^2\bar{R}_{2,i}(z)+O(\ve^4),
\end{eqnarray*}
where $\bar{R}_{1,i}(z)=2\bar{w}_i(z)v_i^{(1)}(z)$,  $\bar{R}_{2,i}=2\bar{w}_i(z)v_i^{(2)}(z)$.
Further, we recall that by (\ref{4.4}) we have

\begin{equation}
\frac{1}{V(p_i+z)}-\frac{1}{V(p_i)}=\frac{1}{V(p_i)^2}(-\xi_{\sigma,i}^2R_1(z)-\xi_{\sigma,i}^2R_2(z)+h.o.t).
\end{equation}

By the reduced problem in Section 5, we have
\[\xi_{\sigma,i}\bar{R}_{2,i}-\xi_{\sigma,i}^2 \bar{w}_i^2 R_{2}
+O(\xi_\sigma\ep^3)
\perp\frac{\partial \bar{w}_i}{\partial\alpha}.
\]
Therefore we can add another contribution $\bar{v}_i^{(3)}$ to the solution such that
$\ve^2\bar{v}_i^{(3)}$ satisfies
\[
\left\{\begin{array}{l}
\tilde{L}^{(i)} v=
\ep^2\bar{R}_{2,i}-\xi_{\sigma,i} \bar{w}_i^2 R_{2}
+O(\xi_\sigma\ep^3)\quad\mbox{ in }\Om,
\\[2mm]
\frac{\partial v}{\partial\nu }=0
\quad\mbox{ on }\partial\Om,
\end{array}\right.
\]
where
\[\tilde{L}^{(i)}\phi=\Delta\phi-\phi
+\frac{2\xi_{\sigma,i}(\bar{w}_i
+\ep^2\bar{v}_i^{(2)})\phi}{
T\left[\xi_{\sigma,i}^2(\bar{w}_i
+\ep^2\bar{v}_i^{(2)}
)^2\right]}
\]\[
-\frac{\xi_{\sigma,i}^2(\bar{w}_i
+\ep^2\bar{v}_{i}^{(2)})^2}{
T\left[\xi_{\sigma,i}^2(\bar{w}_i
+\ep^2\bar{v}_{i}^{(2)})^2\right]^2}\,
T\left[2\xi_{\sigma,i}(\bar{w}_i
+\ep^2\bar{v}_{i}^{(2)})\phi\right]
.\]
Setting
\[
\bar{v}^{(3)}=\sum_{i=1}^k \xi_{\sigma,i}\bar{v}_i^{(3)}
\]
and adding this part to the solution will
cancel out the odd terms (with respect to $\alpha=0$) of order $\ep^2$ in the activator equation and we get
\[
S_1(\bar{w}+\ep^3\bar{v}^{(1)}+\ep^2(\bar{v}^{(2)}+\bar{v}^{(3)})+O(\ep^4),
T[\bar{w}+\ep^3\bar{v}^{(1)}+\ep^2(\bar{v}^{(2)}+\bar{v}^{(3)})+O(\ep^4)])\\
\]\[
=\sum_{i=1}^k 2\xi_{\sigma,i}
\ep^3 \bar{w}_i\bar{v}_i^{(1)}
+\sum_{i=1}^k\xi_{\sigma,i}^2 \bar{w}_{i}^2(-\frac{1}{(V(p_i))^2}
\frac{1}{2}V''(p_i)(x-p_i)^2)
+O(\ve^4).
\]
Taking the derivative $\frac{\partial}{\partial\alpha}$
in this relation near $x=p_i$, we compute
\[
\frac{\partial}{\partial\alpha}
S_1(\bar{w}+\ep^3\bar{v}^{(1)}
+\ep^2(\bar{v}^{(2)}+\bar{v}^{(3)})+O(\ep^4),
T[\bar{w}+\ep^3\bar{v}^{(1)}
+\ep^2(\bar{v}^{(2)}+\bar{v}^{(3)})+O(\ep^4)])
\]\[=
2\sum_{i=1}^k\xi_{\sigma,i}\frac{\partial \bar{w}_i(z)}{\partial \alpha}
(\ve^3\bar{v}_i^{(1)}(z)
+\ve^2(\bar{v}_i^{(2)}(z)+\bar{v}^{(3)}(z)))
\]\[
+2\sum_{i=1}^k\xi_{\sigma,i}\bar{w}_i(z)
\frac{\partial}{\partial\alpha}(
\ve^3\bar{v}_i^{(1)}(z)
+\ve^2(\bar{v}_i^{(2)}(z)+\bar{v}^{(3)}(z)))
\]\[
+\sum_{i=1}^k\xi_{\sigma,i}^2
2\bar{w}_i\frac{\partial \bar{w}_i}{\partial\alpha}
(\frac{1}{V(x)}-\frac{1}{V(p_i)})
\]\[
+\sum_{i=1}^k\xi_{\sigma,i}^2
\bar{w}_i^2
\frac{\partial}{\partial\alpha}
(\frac{1}{V(x)}-\frac{1}{V(p_i)})+O(\ve^4)
\]\[
= \sum_{i=1}^k 2\xi_{\sigma,i}
\ep^3
\left(
\frac{\partial\bar{w}_i}{\partial\alpha}
\bar{v}_i^{(1)}
+
\bar{w}_i
\frac{\partial\bar{v}_i^{(1)}}{\partial\alpha}
\right)
\]\[
+\sum_{i=1}^k\xi_{\sigma,i}^2 \bar{w}_{i}^2(-\frac{1}{(V(p_i))^2}
\frac{1}{2}V''(p_i)(x-p_i)^2)
+O(\ve^4).
\]
On the other hand,
\[
\frac{\partial}{\partial\alpha}
S_1(\bar{w}+\ep^3\bar{v}^{(1)}
+\ep^2(\bar{v}^{(2)}+\bar{v}^{(3)})+O(\ep^4),
T[\bar{w}+\ep^3\bar{v}^{(1)}
+\ep^2(\bar{v}^{(2)}+\bar{v}^{(3)})+O(\ep^4)])
\]\begin{equation}
=
L^{(i)}\left[\xi_{\sigma,i}(\frac{\partial \bar{w}_i}{\partial\alpha}
+\ep^3\frac{\partial\bar{v}_{i}^{(1)}}{\partial\alpha}
+\ep^2\frac{\partial\bar{v}_{i}^{(2)}}{\partial\alpha}
+\ep^2\frac{\partial\bar{v}^{(3)}}{\partial\alpha}
)\right]
+O(\ep^4),
\label{dalphs}
\end{equation}
where
\[L^{(i)}\phi=\Delta\phi-\phi
+\frac{2\xi_{\sigma,i}(\bar{w}_i+\ep^3 \bar{v}_{i}^{(1)}
+\ep^2\bar{v}_{i}^{(2)}
+\ep^2\bar{v}_{i}^{(3)}
)\phi}{
T\left[\xi_{\sigma,i}^2(\bar{w}_i+\ep^3 \bar{v}_{i}^{(1)}
+\ep^2\bar{v}_{i}^{(2)}
+\ep^2\bar{v}_{i}^{(3)}
)^2\right]}
\]\[
-\frac{\xi_{\sigma,i}^2(\bar{w}_i+\ep^3 \bar{v}_{i}^{(1)}
+\ep^2\bar{v}_{i}^{(2)}
+\ep^2\bar{v}_{i}^{(3)}
)^2}{
T\left[(\bar{w}_i+\ep^3 \bar{v}_{i}^{(1)}
+\ep^2\bar{v}_{i}^{(2)}
+\ep^2\bar{v}_{i}^{(3)})^2\right]^2}\,
T\left[2\xi_{\sigma,i}(\bar{w}_i+\ep^3 \bar{v}_{i}^{(1)}
+\ep^2\bar{v}_{i}^{(2)}
+\ep^2\bar{v}_{i}^{(3)}
)\phi\right]
.\]
Finally, by an expansion similar to (\ref{barwb}) we have
\[
\frac{\partial \bar{v}^{(3)}_i}{\partial\nu}
=O(\ep^3) \quad\mbox{ on }\partial\Om
\]
and
\[
\frac{\partial}{\partial\nu}
\frac{\partial \bar{v}^{(3)}_i}{\partial\alpha}=O(\ep^3)
\quad\mbox{ on }\partial\Om.
\]

\subsection{Expansion of the eigenfunction}
We define the approximate kernels to be
\begin{eqnarray*}
&&\mathcal{K}_{\ve,{\bf p}}:=\operatorname{Span}\{\frac{\partial u_i}{\partial \tau(p_i)}, i=1,\cdots,k\},\\
&&\mathcal{C}_{\ve,{\bf p}}:=\operatorname{Span}\{\frac{\partial u_i}{\partial \tau(p_i)}, i=1,\cdots,k\}.
\end{eqnarray*}
Then we expand the eigenfunction as follows:
\[
\phi_\ep=
\sum_{i=1}^k a_{i,\ve}\left(\frac{\partial \bar{w}_i}{\partial\alpha}
+\ep^3\bar{v}_{eig,i}^{(1)}
+\ep^2\bar{v}_{eig,i}^{(2)}
+\ep^2\frac{\partial \bar{v}_i^{(3)}}{\partial\alpha}\right)
+\phi_{\ep}^\perp
+O(\ep^4),
\]
\begin{equation}
\label{equation1}
:= \sum_{i=1}^k a_{\ve,i}\phi_{i,\ve}+\phi_{\ep}^\perp
+O(\ep^4),
\end{equation}
where $\phi_{\ep}^\perp\in \mathcal{K}_{\ve,{\bf p}}^\perp$.

Suppose that $\|\phi_\ve\|_{H^2(\Omega_\ve)}=1$. Then $|a_{j,\ve}|\leq C$.
Let us put
\begin{equation}\label{e:21a}
{\bf a}^\ve:=(a_{1,\ve},\cdots, a_{k,\ve})^T
\end{equation}
Then for a subsequence
and
\begin{equation}\label{e:22a}
a_{i,0}=\lim_{\ve\to 0}a_{i,\ve}, \ {\bf a}^0:=(a_{1,0},\cdots,a_{k,0}).
\end{equation}

The decomposition of $\phi_\ve$ in (\ref{equation1}) implies that
\begin{equation}
\psi_{\ve}=\sum_{i=1}^ka_{i,\ve}\psi_{i,\ve}+\psi_{\ve}^\perp,
\end{equation}
where $\psi_{i,\ve}$ satisfies
\begin{equation}
\left\{\begin{array}{l}
\Delta \psi_{i,\ve}-\sigma^2(1+\tau \lambda_\ve)\psi_{i,\ve}+2u\phi_{i,\ve}=0 \mbox{ in }\Omega_\ve\\
\frac{\partial \psi_{i,\ve}}{\partial \nu}=0  \mbox{ on }\partial \Omega_\ve
\end{array}
\right.
\end{equation}
which we also write as
$\psi_{i,\ve}=T_{\tau\lambda_\ve}[\phi_{i,\ve}]$,
and $\psi_\ve^\perp$ is given by
\begin{equation}
\left\{\begin{array}{l}
\Delta \psi_\ve^\perp-\sigma^2(1+\tau \lambda_\ve)\psi_{\ve}^\perp+2u\phi_\ve^\perp=0 \mbox{ in }\Omega_\ve\\
\frac{\partial \psi_\ve^\perp}{\partial \nu}=0  \mbox{ on }\partial \Omega_\ve
\end{array}
\right.
\end{equation}
which we also represent as
$\psi_\ve^\perp=T_{\tau\lambda_\ve}[\phi_\ve^\perp]$.

Let us first consider
the leading term of $\phi_\ve$.
 For $\frac{\partial \bar{w}_i}{\partial\alpha}$ we get, using
(\ref{barwb}),
\[
\left\{\begin{array}{l}
\Delta \frac{\partial \bar{w}_i}{\partial \alpha}
- \frac{\partial \bar{w}_i}{\partial \alpha}
+\frac{2 \bar{w}_i \frac{\partial \bar{w}_i}{\partial \alpha}}{T[\bar{w}_i^2]}
-\frac{\bar{w}_i^2}{(T[\bar{w}_i^2])^2}
\frac{\partial}{\partial \alpha}T[\bar{w}_i^2]=O(\ep^4\alpha^3)\quad \mbox{ in }\Om,
\\[2mm]
\frac{\partial}{\partial b}
\frac{\partial \bar{w}_i}{\partial\alpha}=
\frac{1}{R}\frac{\partial{w'(|y|)}}{\partial\alpha}
(\frac{1}{2}g^{(3)}(0)\ep^2\alpha^2
+\frac{1}{6}g^{(4)}(0)\ep^3\alpha^3)
 (1+O(\ep\alpha))
\quad \mbox{ on }\partial\Om.
\end{array}\right.
\]
Therefore, expanding the boundary condition as we did above for $\bar{w}_i$, we define $\bar{v}_{eig,i}^{(1)}$ as the unique solution of
\[
\left\{\begin{array}{l}
\Delta v-v=0 \quad \mbox{ in }B_R,
\\[2mm]
\frac{\partial v}{\partial b} =\frac{1}{6R}
\frac{\partial w'(|y|)}{\partial\alpha} g^{(4)}(0)\ep^3\alpha^3
\quad \mbox{ on }\partial B_R.
\end{array}\right.
\]
Similarly, let $\bar{v}_{eig,i}^{(2)}$ be the unique solution of
\[
\left\{\begin{array}{l}
\Delta v-v=0 \quad \mbox{ in }B_R,
\\[2mm]
\frac{\partial v}{\partial b} =\frac{1}{2R}
\frac{\partial w'(|y|)}{\partial\alpha} g^{(3)}(0)\ep^2\alpha^2
\quad \mbox{ on }\partial B_R.
\end{array}\right.
\]
Let us compare this with the derivative $\frac{\partial\bar{v}^{(1)}}{\partial\alpha}$ which satisfies
\[
\left\{\begin{array}{l}
\Delta v-v=0
\\[2mm]
\frac{\partial v}{\partial b} =\frac{\partial}{\partial\alpha}[\frac{1}{6R}
w'(|y|) g^{(4)}(0)\ep^3\alpha^3]
\quad \mbox{ on }\partial B_R.
\end{array}\right.
\]
Further,
$\frac{\partial\bar{v}^{(2)}}{\partial\alpha}$ solves
\[
\left\{\begin{array}{l}
\Delta v-v=0 \quad \mbox{ in }B_R,
\\[2mm]
\frac{\partial v}{\partial b} =\frac{\partial}{\partial \alpha}[\frac{1}{2R}w'(|y|) g^{(3)}(0)\ep^2\alpha^2]
\quad \mbox{ on }\partial B_R.
\end{array}\right.
\]

Using (\ref{dalphs}), we get
\[
L\left[
\sum_{i=1}^k a_{i,\ve}\left(\frac{\partial \bar{w}_i}{\partial\alpha}
+\ep^3\bar{v}_{eig,i}^{(1)}
+\ep^2\bar{v}_{eig,i}^{(2)}
+\ep^2\frac{\partial \bar{v}_i^{(3)}}{\partial\alpha}\right)
\right]
\]
\[
-L\left[
\sum_{i=1}^k a_{i,\ve}\left(\frac{\partial \bar{w}_i}{\partial\alpha}
+\ep^3\frac{\partial\bar{v}_{i}^{(1)}}{\partial\alpha}
+\ep^2\frac{\partial\bar{v}_{i}^{(2)}}{\partial\alpha}
+\ep^2\frac{\partial \bar{v}_i^{(3)}}{\partial\alpha}\right)
\right]
\]\[
=2\sum_{i=1}^k
a_{i,\ve}
\ep^3
\bar{w}_i
\bar{v}_{R,i}^{(1)}
\]\[
+\sum_{i=1}^ka_{i,\ve}
(\frac{\bar{w}_i^2}{v^2}\frac{\partial v}{\partial \tau(p_i)}-
\frac{1}{\xi_{\sigma,i}}
\frac{u^2}{v^2}\psi_{i,\ve})+O(\ve^4),
\]
where
\[
L\phi=\Delta\phi-\phi
+\frac{2(\bar{w}+\ep^3 \bar{v}^{(1)}
+\ep^2\bar{v}^{(2)}
+\ep^2\bar{v}^{(3)}
)\phi}{
T\left[(\bar{w}+\ep^3 \bar{v}^{(1)}
+\ep^2\bar{v}^{(2)}
+\ep^2\bar{v}^{(3)}
)^2\right]}
\]\[
-\frac{(\bar{w}+\ep^3 \bar{v}^{(1)}
+\ep^2\bar{v}^{(2)}
+\ep^2\bar{v}^{(3)}
)^2}{
T\left[(\bar{w}+\ep^3 \bar{v}^{(1)}
+\ep^2\bar{v}^{(2)}
+\ep^2\bar{v}^{(3)})^2\right]^2}\,
T\left[2\xi_{\sigma,i}(\bar{w}+\ep^3 \bar{v}^{(1)}
+\ep^2\bar{v}^{(2)}
+\ep^2\bar{v}^{(3)}
)\phi\right]
\]
and the remainder $\bar{v}_{R,i}^{(1)}$
is given by the difference of the previous two contributions as follows:

$\bar{v}_{R,i}^{(1)}=\bar{v}_{eig,i}^{(1)}-\frac{\partial \bar{v}^{(1)}}{\partial\alpha}$ which satisfies
\begin{equation}
\label{v1r}
\left\{\begin{array}{l}
\Delta v-v=0 \quad\mbox{ in } B_R,
\\[2mm]
\frac{\partial v}{\partial b}= -\frac{1}{6R}w'(|y|)g^{(4)}(0)3 \alpha^2\,(1+O(\ep\alpha))\quad \mbox{ on }\partial B_R.
\end{array}\right.
\end{equation}
We note that $\bar{v}_{R,i}^{(1)}$ is an even
function around $\alpha=0$.

Substituting the decompositions of $\phi_\ve$ and $\psi_\ve$ into (\ref{equ1}),
we have
\begin{eqnarray}
\label{equation3}
&&
L\left[\sum_{i=1}^k a_{i,\ve}\left(\frac{\partial \bar{w}_i}{\partial\alpha}
+\ep^3\bar{v}_{eig,i}^{(1)}
+\ep^2\bar{v}_{eig,i}^{(2)}
+\ep^2\frac{\partial \bar{v}_i^{(3)}}{\partial\alpha}\right)
+\phi_\ep^\perp
\right]
\nonumber
\\[2mm]
&& =\sum_{i=1}^k 2
\ep^3
\bar{w}_i
\bar{v}_{R,i}^{(1)}
\nonumber
\\[2mm]
&&+\frac{1}{\xi_\sigma}\sum_{i=1}^ka_{i,\ve}(\frac{u_i^2}{v^2}\frac{\partial v}{\partial \tau(p_i)}-\frac{u^2}{v^2}\psi_{i,\ve})\nonumber\\
&&
+\Delta \phi_{\ve}^\perp-\phi_{\ve}^\perp+\frac{2u}{v}\phi_{\ve}^\perp-\frac{u^2}{v^2}\psi_{\ve}^\perp
-\lambda_\ve\phi_{\ve}^\perp\nonumber\\
&& =\frac{1}{\xi_\sigma}\lambda_\ve\sum_{i=1}^ka_{i,\ve}\frac{\partial u_i}{\partial \tau(p_i)}+O(\ve^4).
\end{eqnarray}
Set
\begin{equation*}
I_3=\sum_{i=1}^k 2
\ep^3
\bar{w}_i
\bar{v}_{R,i}^{(1)},
\end{equation*}
\begin{equation*}
I_4=\frac{1}{\xi_\sigma}\sum_{i=1}^ka_{i,\ve}(\frac{u_i^2}{v^2}\frac{\partial v}{\partial \tau(p_i)}-\frac{u^2}{v^2}\psi_{i,\ve})
\end{equation*}
and
\begin{equation*}
I_5=\Delta \phi_{\ve}^\perp-\phi_{\ve}^\perp+\frac{2u}{v}\phi_{\ve}^\perp-\frac{u^2}{v^2}\psi_{\ve}^\perp-\lambda\phi_{\ve}^\perp.
\end{equation*}

We first compute
\begin{eqnarray*}
I_4&&=\frac{1}{\xi_\sigma}\sum_{i=1}^ka_{i,\ve}(\frac{u_i^2}{v^2}\frac{\partial v}{\partial \tau(p_i)}-\frac{u^2}{v^2}\psi_{i,\ve})+O(\ve^4)\\
&&=\frac{1}{\xi_\sigma}\sum_{i=1}^ka_{i,\ve}(\frac{u_i^2}{v^2}\frac{\partial v}{\partial \tau(p_i)}-\frac{u_i^2}{v^2}\psi_{i,\ve})-\frac{1}{\xi_\sigma}\sum_{i=1}^k\sum_{j\neq i}a_{i,\ve}\frac{u_j^2}{v^2}\psi_{i,\ve}+O(\ve^4).
\end{eqnarray*}

Since
\begin{eqnarray*}
&&\frac{1}{\xi_\sigma }\sum_{i=1}^k\sum_{|i-j|\geq 2}a_{i,\ve}\frac{u_j^2}{v^2}\psi_{i,\ve}\\
&&=\xi_{\sigma}\sum_{i=1}^k\sum_{|i-j|\geq 2}|\nabla_{p_i}G_{\sqrt{D}}(\sigma_\lambda p_i,\sigma_\lambda p_j)|(1+o(1))\\
&&=O(\xi_{\sigma}\sigma \log\ve (\frac{\ve D}{\xi_\sigma }\log \frac{\ve D}{\xi_\sigma })^2)\\
&&=O(\sigma (\frac{\ve D}{\xi_\sigma }\log \frac{\ve D}{\xi_\sigma })^2),
\end{eqnarray*}
we can estimate $I_4$ as follows:
\begin{eqnarray*}
I_4&&=\frac{1}{\xi_\sigma}\sum_{i=1}^k\sum_{j=1}^ka_{i,\ve}\frac{u_j^2}{v^2}(\frac{\partial v}{\partial \tau(p_j)}\delta_{ij}-\psi_{i,\ve}))+O(\ve^4)\\
&&=-\frac{1}{\xi_\sigma}\sum_{i=1}^k\sum_{|i-j|=1}a_{i,\ve}\frac{u_j^2}{v^2}\psi_{i,\ve}+O(\ve^4)+O(\sigma^2\lambda_\ve)+O(\sigma (\frac{\ve D}{\xi_\sigma }\log \frac{\ve D}{\xi_\sigma })^2),
\end{eqnarray*}
where we use the equation satisfied by $\psi_{i,\ve}$.

\subsection{Expansion of the small eigenvalues}

Multiplying both sides of (\ref{equation3}) by $\frac{\partial u_l}{\partial \tau(p_l)}$ and integrating over $\Omega_\ve$, we have
\begin{eqnarray*}
r.h.s&&=\lambda_\ve\frac{1}{\xi_\sigma}\sum_{i=1}^ka_{i,\ve}\int_{\Omega_\ve}\frac{\partial u_i}{\partial \tau(p_i)}\frac{\partial u_l}{\partial \tau(p_l)}\,dx\\
&&=\lambda_\ve\xi_\sigma a_{l,\ve}\int_{\R^2_+}(\frac{\partial w}{\partial y_1})^2\,dy(1+o(1)).
\end{eqnarray*}
Using the estimate $I_4$, we have
\begin{eqnarray*}
l.h.s&&=\int_{\Omega_\ve}(\Delta \phi_{\ve}^\perp-\phi_{\ve}^\perp
+\frac{2u}{v}\phi_{\ve}^\perp
-\frac{u^2}{v^2}\psi_{\ve}^\perp
-\lambda_\ep\phi_{\ve}^\perp)\frac{\partial u_l}{\partial \tau(p_l)}\,dx\nonumber\\
&&+\int_{\Omega_\ve}\frac{1}{\xi_\sigma}\sum_{i=1}^ka_{i,\ve}(\frac{u_i^2}{v^2}\frac{\partial v}{\partial \tau(p_i)}-\frac{u^2}{v^2}\psi_{i,\ve})\frac{\partial u_l}{\partial \tau(p_l)}\,dx+O(\ve^4)\\
&&=\int_{\Omega_\ve}\frac{u_l^2}{v^2}
\frac{\partial v}{\partial \tau(p_l)}\phi_{\ve}^\perp\,dx
-\int_{\Omega_\ve}\frac{u^2}{v^2}
\psi_{\ve}^\perp\frac{\partial u_l}{\partial \tau(p_l)}\,dx
-\lambda_\ve\int_{\Omega_\ve}\phi_{\ve}^\perp
\frac{\partial u_l}{\partial \tau(p_l)}\,dx\\
&&+\int_{\Omega_\ve}\frac{1}{\xi_\sigma}\sum_{i=1}^k\sum_{j=1}^ka_{i,\ve}\frac{u_j^2}{v^2}(\frac{\partial v}{\partial \tau(p_j)}\delta_{ij}-\psi_{i,\ve}))\frac{\partial u_l}{\partial \tau(p_l)}\,dx
\\
&&+
\sum_{i=1}^k \int_{\Om_\ve}2a_{i,\ve}
\ve^3 \bar{w}_i \bar{v}_{R,i}^{(1)}
\frac{\partial u_l}{\partial \tau(p_l)}
+O(\ve^4)\\
&&=J_{1,l}+J_{2,l}+J_{3,l}+J_{4,l}+J_{5,l}+O(\ve^4),
\end{eqnarray*}
where $J_{i,l}$ are defined as the integrals in the last equality.
We divide our proof into several steps.

The following lemma contains the key estimates:
\begin{lemma}\label{lemma4}
We have
\begin{eqnarray}
&&J_{1,l}=o(\xi_\sigma^2\sigma^2\frac{\ve D}{\xi_\sigma}),\label{e:29}\\
&&J_{2,l}=o(\xi_\sigma^2\sigma^2\frac{\ve D}{\xi_\sigma}),\label{e:27}\\
&&J_{3,l}=o(\xi_\sigma  \lambda_\ve),\label{e:26}\\
&&J_{4,l}=(\int_{\R^2_+}w^2\,dy\int_{\R^2_+}w^2\frac{\partial w}{\partial y_1}y_1\,dy)\xi_{\sigma}^2(1+o(1))\times \Big[[\sum_{|i-l|=1}\nabla^2_{\tau(p_l)}G_{\sqrt{D}}(\sigma p_i,\sigma p_l)]a_{l,\ve}\nonumber\\
&&
-[a_{l-1,\ve}\nabla^2_{\tau(p_l)}G_{\sqrt{D}}(\sigma p_l,\sigma p_{l-1})+a_{l+1,\ve}\nabla^2_{\tau(p_l)}G_{\sqrt{D}}(\sigma p_l,\sigma p_{l+1})]\Big]
+o(\xi_\sigma^2\sigma^2\frac{\ve D}{\xi_\sigma})
\nonumber\\ &&
J_{5,l}=\ep^3 \xi_{\sigma,l} a_{l,\ve}\frac{3}{4}\nu_1 \frac{\partial^2 }{\partial \tau(\ve p_l)^2}h(\ve p_l)+
o(\xi_\sigma^2\sigma^2\frac{\ve D}{\xi_\sigma}),\\
\label{e:34}
\end{eqnarray}
where $a_{i,\ve}$ has been defined in (\ref{equation1}).
\end{lemma}

\begin{proof}
We first study the asymptotic behaviour of $\psi_{j,\ve}$.

Note that for $l\neq k$, we have
\begin{eqnarray}\label{e:23}
\psi_{k,\ve}(p_l)&&=\int_{\Omega_\ve}G_{\sqrt{D}}(\sigma_\lambda p_l,\sigma_\lambda y)2u\phi_{k,\ve}(y)\,dy\\
&&=\nabla_{\tau(p_k)}G_{\sqrt{D}}(\sigma_\lambda p_l,\sigma_\lambda p_k)\int_{\R^2_+}2\xi_{\sigma,k}^2w\frac{\partial w}{\partial y_1}y\,dy+O(\xi_{\sigma}^2\sigma^2\frac{\ve D}{\xi_\sigma})\nonumber\\
&&=\nabla_{\tau(p_k)}G_{\sqrt{D}}(\sigma p_l,\sigma p_k)\int_{\R^2_+}2\xi_{\sigma,k}^2w\frac{\partial w}{\partial y_1}y\,dy(1+O(\lambda_\ve \log \ve))+o(\xi_{\sigma}^2\sigma^2\frac{\ve D}{\xi_\sigma}).\nonumber
\end{eqnarray}
Next we compute $\psi_{l,\ve}-\frac{\partial v}{\partial \tau(p_l)}$ near $p_l$:
\begin{eqnarray*}
v(x)&&=\int_{\Omega_\ve}G_{\sqrt{D}}(\sigma x,\sigma y)u^2(y)\,dy\\
&&=\int_{\Omega_\ve}\frac{1}{\pi}\log\frac{1}{\sigma|x-y|}u_l^2(y)\,dy+\int_{\Omega_\ve}\tilde{H}(\sigma x,\sigma y)u^2_l(y)\,dy\\
&&+\sum_{i\neq l}\int_{\Omega_\ve}G_{\sqrt{D}}(\sigma x,\sigma y)u_i^2(y)\,dy+O(\ve^4),
\end{eqnarray*}
so
\begin{eqnarray*}
\frac{\partial v}{\partial \tau(p_l)}(x)&&=\sum_{|i-l|=1}\nabla_{\tau(p_i)}G_{\sqrt{D}}(\sigma x,\sigma p_i)\xi_{\sigma,i}^2\int_{\R^2_+}w^2\,dy+O(\xi_\sigma^2\sigma |\log \ve|(\frac{\ve D}{\xi_\sigma}\log\frac{\ve D}{\xi_\sigma})^2)\\
&&=\sum_{|i-l|=1}\nabla_{\tau(p_i)}G_{\sqrt{D}}(\sigma x,\sigma p_l)\xi_{\sigma,i}^2\int_{\R^2_+}w^2\,dy+o(\xi_{\sigma}^2\sigma^2\frac{\ve D}{\xi_\sigma}),
\end{eqnarray*}
where we have used the relation $e^{-\frac{1}{\sqrt{D}}}\ll\ve$.

Thus
\begin{eqnarray}\label{e:24}
\frac{\partial v}{\partial \tau(p_l)}-\psi_{l,\ve}(p_l)&&=\sum_{|i-l|=1}\nabla_{\tau(p_l)}G_{\sqrt{D}}(\sigma p_l,\sigma p_i)\xi^2_{\sigma,i}\int_{\R^2_+}w^2\,dy\nonumber\\
&&+o(\xi_{\sigma}^2\sigma^2\frac{\ve D}{\xi_\sigma}).
\end{eqnarray}

Combining (\ref{e:23}) and (\ref{e:24}), we have
\begin{eqnarray}\label{equation5}
(\psi_{i,\ve}-\frac{\partial v}{\partial \tau(p_i)}\delta_{il})(p_l)&&=\nabla_{\tau(p_i)}G_{\sqrt{D}}(\sigma p_l,\sigma p_i)(2\xi_{\sigma,i}^2\int_{\R^2_+}w\frac{\partial w}{\partial y_1}y_1\,dy)[\delta_{i,l-1}+\delta_{i,l+1}]\nonumber\\
&&-\delta_{il}\sum_{|m-l|=1}\nabla_{\tau(p_l)}G_{\sqrt{D}}(\sigma p_l,\sigma p_m)(\xi_{\sigma,m}^2\int_{\R^2_+}w^2\,dy)\nonumber\\
&&+o(\xi_{\sigma}^2\sigma^2\frac{\ve D}{\xi_\sigma})+o(\xi_\sigma^2\lambda_\ve).
\end{eqnarray}

Similarly, we have the following:
\begin{eqnarray}\label{e:30}
&&\psi_{i,\ve}(p_l+(y_1,0))-\psi_{i,\ve}(p_l)\\
&&=\int_{\Omega_\ve}(G_{\sqrt{D}}(\sigma_\lambda(p_l+y),\sigma_\lambda x)-G_{\sqrt{D}}(\sigma_\lambda p_l,\sigma_\lambda x))2u\phi_{i,\ve}(x)\,dx\nonumber\\
&&=2\xi_{\sigma,i}^2\nabla_{\tau(p_i)}\nabla_{\tau(p_l)}G_{\sqrt{D}}(\sigma_\lambda p_l,\sigma_\lambda p_i)y_1\int_{\R^2_+}w\frac{\partial w}{\partial x_1}x_1\,dx+O(\sigma^3\xi_\sigma^2\frac{\ve D}{\xi_\sigma}\log\frac{\ve D}{\xi_\sigma} y^2)\nonumber\\
&&=2\xi_{\sigma,i}^2\nabla_{\tau(p_i)}\nabla_{\tau(p_l)}G_{\sqrt{D}}(\sigma_\lambda p_l,\sigma_\lambda p_i)y_1\int_{\R^2_+}w\frac{\partial w}{\partial x_1}x_1\,dx[\delta_{l,i-1}+\delta_{l,i+1}]\nonumber\\
&&
+O(\sigma^3\xi_\sigma^2\frac{\ve D}{\xi_\sigma}\log\frac{\ve D}{\xi_\sigma} y^2)
+O(\xi_\sigma^2\sigma^2\log\ve (\frac{\ve D}{\xi_\sigma}\log\frac{\ve D}{\xi_\sigma})^2y)\nonumber\\
&&=2\xi_{\sigma,i}^2\nabla_{\tau(p_i)}\nabla_{\tau(p_l)}G_{\sqrt{D}}(\sigma p_l,\sigma p_i)y_1\int_{\R^2_+}w\frac{\partial w}{\partial x_1}x_1\,dx[\delta_{l,i-1}+\delta_{l,i+1}]\nonumber\\
&&+O(\sigma^3\xi_\sigma^2\frac{\ve D}{\xi_\sigma}\log\frac{\ve D}{\xi_\sigma} y^2)+O(\xi_\sigma^2\sigma^2\log\ve (\frac{\ve D}{\xi_\sigma}\log\frac{\ve D}{\xi_\sigma})^2y)+o(\xi_\sigma^2 \lambda_\ve y)\nonumber
\end{eqnarray}
for $i\neq l$
and
\begin{eqnarray}\label{e:31}
&&(\psi_{l,\ve}-\frac{\partial v}{\partial \tau(p_l)})(p_l+(y_1,0))-(\psi_{l,\ve}-\frac{\partial v}{\partial \tau(p_l)})(p_l)\\
&&=\sum_{|i-l|=1}\xi_{\sigma,i}^2\nabla_{\tau(p_i)}\nabla_{\tau(p_l)}G_{\sqrt{D}}(\sigma p_i,\sigma p_l)y_1\int_{\R^2_+}w^2\,dx\nonumber\\
&&+O(\sigma^3\xi_\sigma^2\frac{\ve D}{\xi_\sigma}\log\frac{\ve D}{\xi_\sigma} y^2)+O(\xi_\sigma^2\sigma^2\log\ve (\frac{\ve D}{\xi_\sigma}\log\frac{\ve D}{\xi_\sigma})^2y)+o(\xi_\sigma^2 \lambda_\ve y).\nonumber
\end{eqnarray}
Thus we have for $J_{4,l}$,
\begin{eqnarray*}
J_{4,l}&&=\frac{1}{\xi_\sigma}\sum_{i,j=1}^k\int_{\Omega_\ve}a_{i,\ve}\frac{u_j^2}{v^2}(\frac{\partial v}{\partial \tau(p_j)}\delta_{ij}-\psi_{i,\ve})\frac{\partial u_l}{\partial \tau(p_l)}\,dx\\
&&=\frac{1}{\xi_\sigma}\sum_{i,j=1}^k\int_{\Omega_\ve}a_{i,\ve}\frac{u_j^2}{v^2}(\frac{\partial v}{\partial \tau(p_j)}\delta_{ij}-\psi_{i,\ve})(p_l)\frac{\partial u_l}{\partial \tau(p_l)}\,dx\\
&&+
\frac{1}{\xi_\sigma}\sum_{i,j=1}^k\int_{\Omega_\ve}a_{i,\ve}\frac{u_j^2}{v^2}[(\frac{\partial v}{\partial \tau(p_j)}\delta_{ij}-\psi_{i,\ve})(x)-(\frac{\partial v}{\partial \tau(p_j)}\delta_{ij}-\psi_{i,\ve})(p_l)]\frac{\partial u_l}{\partial \tau(p_l)}\,dx\\
&&=J_{6,l}+J_{7,l}.
\end{eqnarray*}
For $J_{6,l}$, we have from (\ref{equation5})
\begin{eqnarray*}
J_{6,l}&&=\frac{1}{\xi_\sigma}\sum_{i,j=1}^k\int_{\Omega_\ve}a_{i,\ve}\frac{u_j^2}{v^2}(\frac{\partial v}{\partial \tau(p_j)}\delta_{ij}-\psi_{i,\ve})(p_l)\frac{\partial u_l}{\partial \tau(p_l)}\,dx\\
&&=\sum_{i}^k\frac{1}{\xi_\sigma}(\frac{\partial v}{\partial \tau(p_l)}\delta_{il}-\psi_{i,\ve})(p_l)\int_{\Omega_\ve}a_{i,\ve}\frac{u_l^2}{v^2}\frac{\partial u_l}{\partial \tau(p_l)}\,dx\\
&&+\sum_{i=1}^k\sum_{j\neq l}\frac{1}{\xi_\sigma}(\frac{\partial v}{\partial \tau(p_j)}\delta_{ij}-\psi_{i,\ve})(p_l)\int_{\Omega_\ve}a_{i,\ve}\frac{u_j^2}{v^2}\frac{\partial u_l}{\partial \tau(p_l)}\,dx=o(\xi_\sigma^2\sigma^2\frac{\ep D}{\xi_\sigma}).
\end{eqnarray*}

Similarly, using (\ref{e:30}), (\ref{e:31}), (\ref{e:32}) and (\ref{e:33}), we get
\begin{eqnarray*}
J_{7,l}&&=\frac{1}{\xi_\sigma}\sum_{i,j=1}^k\int_{\Omega_\ve}a_{i,\ve}\frac{u_j^2}{v^2}[(\frac{\partial v}{\partial \tau(p_j)}\delta_{ij}-\psi_{i,\ve})(x)-(\frac{\partial v}{\partial \tau(p_j)}\delta_{ij}-\psi_{i,\ve})(p_l)]\frac{\partial u_l}{\partial \tau(p_l)}\,dx\\
&&=\frac{1}{\xi_\sigma}\sum_{i=1}^ka_{i,\ve}\int_{\Omega_\ve}\frac{u_i^2}{v^2}[(\frac{\partial v}{\partial \tau(p_i)}-\psi_{i,\ve})(x)-(\frac{\partial v}{\partial \tau(p_i)}-\psi_{i,\ve})(p_l)]\frac{\partial u_l}{\partial \tau(p_l)}\,dx\\
&&-\sum_{i=1}^k\sum_{j\neq i}
\frac{1}{\xi_\sigma}\int_{\Omega_\ve}a_{i,\ve}\frac{u_j^2}{v^2}[\psi_{i,\ve}(x)-\psi_{i,\ve}(p_l)]\frac{\partial u_l}{\partial \tau(p_l)}\,dx\\
&&=\frac{1}{\xi_\sigma}\int_{\Omega_\ve}a_{l,\ve}\frac{u_l^2}{v^2}[(\frac{\partial v}{\partial \tau(p_l)}-\psi_{l,\ve})(x)-(\frac{\partial v}{\partial \tau(p_l)}-\psi_{l,\ve})(p_l)]\frac{\partial u_l}{\partial \tau(p_l)}\,dx\\
&&-\frac{1}{\xi_\sigma}\sum_{|i-l|=1}\int_{\Omega_\ve}a_{i,\ve}\frac{u_l^2}{v^2}[\psi_{i,\ve}(x)-\psi_{i,\ve}(p_l)]\frac{\partial u_l}{\partial \tau(p_l)}\,dx\\
&&+o(\xi_\sigma^2\sigma^2\frac{\ve D}{\xi_\sigma})\\
&&=(\int_{\R^2_+}w^2\,dy\int_{\R^2_+}w^2\frac{\partial w}{\partial y_1}y_1\,dy)[\sum_{|i-l|=1}\xi_{\sigma,i}^2\nabla^2_{\tau(p_l)}G_{\sqrt{D}}(\sigma p_i,\sigma p_l)]a_{l,\ve}\\
&&+(2\int_{\R^2_+}w\frac{\partial w}{\partial y_1}y_1\,dy\int_{\R^2_+}w^2\frac{\partial w}{\partial y_1}y_1\,dy)\\
&&\times [\xi_{\sigma,l-1}^2a_{l-1,\ve}\nabla^2_{\tau(p_l)}G_{\sqrt{D}}(\sigma p_l,\sigma p_{l-1})+\xi_{\sigma,l+1}^2a_{l+1,\ve}\nabla^2_{\tau(p_l)}G_{\sqrt{D}}(\sigma p_l,\sigma p_{l+1})]\\
&&+o(\xi_\sigma^2\sigma^2\frac{\ve D}{\xi_\sigma})\\
&&=(\int_{\R^2_+}w^2\,dy\int_{\R^2_+}w^2\frac{\partial w}{\partial y_1}y_1\,dy)\xi_{\sigma}^2(1+o(1))\times \Big[[\sum_{|i-l|=1}\nabla^2_{\tau(p_l)}G_{\sqrt{D}}(\sigma p_i,\sigma p_l)]a_{l,\ve}\\
&&
-[a_{l-1,\ve}\nabla^2_{\tau(p_l)}G_{\sqrt{D}}(\sigma p_l,\sigma p_{l-1})+a_{l+1,\ve}\nabla^2_{\tau(p_l)}G_{\sqrt{D}}(\sigma p_l,\sigma p_{l+1})]\Big]+o(\xi_\sigma^2\sigma^2\frac{\ve D}{\xi_\sigma}),
\end{eqnarray*}
where we have used the relation
\begin{equation*}
2\int_{\R^2_+}w\frac{\partial w}{\partial y_1}y_1\,dy=-\int_{\R^2_+}w^2\,dy.
\end{equation*}

Combining the above two estimates, we have
\begin{eqnarray*}
J_{4,l}&&=(\int_{\R^2_+}w^2\,dy\int_{\R^2_+}w^2\frac{\partial w}{\partial y_1}y_1\,dy)\xi_{\sigma}^2(1+o(1))\times \Big[[\sum_{|i-l|=1}\nabla^2_{\tau(p_l)}G_{\sqrt{D}}(\sigma p_i,\sigma p_l)]a_{l,\ve}\\
&&
-[a_{l-1,\ve}\nabla^2_{\tau(p_l)}G_{\sqrt{D}}(\sigma p_l,\sigma p_{l-1})+a_{l+1,\ve}\nabla^2_{\tau(p_l)}G_{\sqrt{D}}(\sigma p_l,\sigma p_{l+1})]\Big]+o(\xi_\sigma^2\sigma^2\frac{\ve D}{\xi_\sigma}),
\end{eqnarray*}
and this proves (\ref{e:34}).

Using
\[
\frac{1}{R}\,\frac{\partial w}{\partial\alpha}=w'(|y|)(1+O(\ep\alpha))
\]
and
\[
\frac{1}{R^4}g^{(4)}(0)=\frac{\partial^2 }{\partial \tau(\ve p_i)^2}h(\ve p_i)\,(1+O(\ep|p_i|)),
\]
we can evaluate the integral $J_{5,l}$.
A computation analogous to (\ref{e:18}) gives
\begin{eqnarray}\label{e:18a}
J_{5,l}&=&
\sum_{i=1}^k 2\ve^3 a_{\ve,i}\int_{\Om_\ep}
\bar{w}_i\bar{v}_{R,i}^{(1)}\frac{\partial u_l}{\partial \tau(p_l)}
\nonumber \\
& =& \xi_{\sigma,l} \ve^3 a_{\ep,l}
\int_{B_{R}}2w
\bar{v}_{R}^{(1)}\frac{1}{r}\frac{\partial w }{\partial \alpha}r\,(1+O(\ep\alpha))\,dr\,d\alpha
\nonumber \\
&
=&\xi_{\sigma,l} \ve^3 a_{\ep,l}\int_{B_R}-\frac{1}{r}[(\Delta -1)\frac{\partial w}{\partial \alpha}]\bar{v}_{R,l}^{(1)}\,r\,(1+O(\ep\alpha))\,dr\,d\alpha\nonumber\\
&=&\xi_{\sigma,l}\ve^3 a_{\ep,l}\int_{\partial B_R}-\frac{1}{R}\left(\frac{\partial w}{\partial \alpha}\frac{\partial \bar{v}_{R,l}^{(1)}}{\partial b}-\bar{v}_{R,l}^{(1)}\frac{\partial }{\partial b}\frac{\partial w}{\partial \alpha}\right)\,R\,(1+O(\ep\alpha))\,d\alpha\nonumber\\
&=&\ep^3\xi_{\sigma,l} a_{\ep,l}\int_{\R}(w')^2\frac{1}{6R^2}g^{(4)}(0) \frac{3y_1^2}{R^2} \,(1+O(\ep y_1))\,d y_1\nonumber\\
&=&\ep^3\xi_{\sigma,l} a_{\ep,l}\frac{3}{2R^4}\nu_1 g^{(4)}(0)
+o(\xi_\sigma^2\sigma^2\frac{\ve D}{\xi_\sigma})\nonumber\\
&=&\ep^3\xi_{\sigma,l} a_{\ep,l}\frac{3}{2}\nu_1 \frac{\partial^2 }{\partial \tau(\ve p_l)^2}h(\ve p_l)
+o(\xi_\sigma^2\sigma^2\frac{\ve D}{\xi_\sigma}),
\end{eqnarray}
where the constant $\nu_1>0$ has been defined in (\ref{nu1}).
Thus
\begin{eqnarray*}
J_{5,l}
&&=
\ep^3 \xi_{\sigma,l} a_{l,\ep} \frac{3}{2}\nu_1 \frac{\partial^2 }{\partial \tau(\ve p_l)^2}h(\ve p_l)+
o(\xi_\sigma^2\sigma^2\frac{\ve D}{\xi_\sigma}).
\end{eqnarray*}

From the estimates on $I_3$ and $I_4$, we know
that
\begin{equation}\label{e:28}
\|\phi_{\ve}^\perp\|_{H^2(\Omega_\ve)}\leq C\|I_3\|_{L^2(\Omega_\ve)}\leq C\xi_\sigma \sigma \frac{\ve D}{\xi_\sigma}\log\frac{\ve D}{\xi_\sigma}.
\end{equation}
Next by the definition of $J_{1,l}$, using (\ref{e:28})
\begin{eqnarray}
J_{1,l}&&=\int_{\Omega_\ve}\frac{u_l^2}{v^2}\frac{\partial v}{\partial \tau(p_l)}\phi_{\ve}^\perp \,dx\nonumber\\
&&=O(\xi_{\sigma}^2\sigma \frac{\ve D}{\xi_\sigma}\log\frac{\ve D}{\xi_\sigma})\|\phi_{\ve}^\perp\|_{H^2(\Omega_\ve)}\nonumber\\
&&=o(\xi_\sigma^2\sigma^2\frac{\ve D}{\xi_\sigma}).
\end{eqnarray}
We have obtained (\ref{e:29}).
Next we estimate $J_{2,l}$:
\begin{eqnarray*}
J_{2,l}&&=-\int_{\Omega_\ve}\frac{u^2}{v^2}\psi_{\ve}^\perp\frac{\partial u_l}{\partial \tau(p_l)}\,dx\\
&&=-\int_{\Omega_\ve}\frac{u_l^2}{v^2}\psi_{\ve}^\perp \frac{\partial u_l}{\partial \tau(p_l)}\,dx+O(\ve^4)\\
&&=-\int_{\Omega_\ve}\frac{u_l^2}{v^2}\psi_{\ve}^\perp(p_l)\frac{\partial u_l}{\partial \tau(p_l)}\,dx-\int_{\Omega_\ve}\frac{u_l^2}{v^2}\frac{\partial u_l}{\partial \tau(p_l)}(\psi_{\ve}^\perp(x)-\psi_{\ve}^\perp(p_l))\,dx+O(\ve^4)\\
&&=-J_{8,l}-J_{9,l}+O(\ve^4).
\end{eqnarray*}

By the equation satisfied by $\psi_{\ve}^\perp$, we have
\begin{eqnarray*}
\psi_{\ve}^\perp(p_l)&&=\int_{\Omega_\ve}G_{\sqrt{D}}(\sigma_\lambda p_l,\sigma_\lambda x)2u\phi_{\ve}^\perp(x)\,dx\\
&&=O(\xi_\sigma \sigma \frac{\ve D}{\xi_\sigma}\log\frac{\ve D}{\xi_\sigma}).
\end{eqnarray*}

Further,
\begin{eqnarray*}
\psi_{\ve}^\perp(p_l+y)-\psi_{\ve}^\perp(p_l)&&=\int_{\Omega_\ve}[G_{\sqrt{D}}(\sigma_\lambda(p_l+y),\sigma_\lambda x)-G_{\sqrt{D}}(\sigma_\lambda p_l,\sigma_\lambda x)]2u\phi_{\ve}^\perp(x)\,dx\\
&&=\int_{\Omega_\ve}[G_{\sqrt{D}}(\sigma_\lambda(p_l+y),\sigma_\lambda x)-G_{\sqrt{D}}(\sigma_\lambda p_l,\sigma_\lambda x)]2\sum_{j=1}^ku_j\phi_{\ve}^\perp(x)\,dx+O(\ve^4)\\
&&=O(\xi_\sigma\sigma^3\frac{\ve D}{\xi_\sigma}\log\frac{\ve D}{\xi_\sigma}y)
\end{eqnarray*}

We have by the estimate for $\frac{\partial v}{\partial \tau(p_l)}$,
\begin{eqnarray*}
J_{8,l}&&=\psi_{\ve}^\perp(p_l)\int_{\Omega_\ve}\frac{u_l^2}{v^2}\frac{\partial u_l}{\partial \tau(p_l)}\,dx\\
&&=\psi_{\ve}^\perp(p_l)(\int_{\Omega_\ve}\frac{2}{3}\frac{u_l^3}{v^3}\frac{\partial v}{\partial \tau(p_l)}+O(\ve^4))\\
&&=O(\xi_\sigma^3\sigma^2(\frac{\ve D}{\xi_\sigma}\log\frac{\ve D}{\xi_\sigma})^2)=o(\xi_\sigma^2\sigma^2\frac{\ve D}{\xi_\sigma}).
\end{eqnarray*}
and
\begin{eqnarray*}
J_{9,l}&&=\int_{\Omega_\ve}\frac{u_l^2}{v^2}\frac{\partial u_l}{\partial \tau(p_l)}(\psi_{\ve}^\perp(x)-\psi_{\ve}^\perp(p_l))\,dx\\
&&=O(\xi_{\sigma}^2\sigma^3\frac{\ve D}{\xi_\sigma}\log\frac{\ve D}{\xi_\sigma})=o(\xi_{\sigma}^2\sigma^2\frac{\ve D}{\xi_\sigma})
\end{eqnarray*}
Thus we have
\begin{equation}
J_{2,l}=o(\xi_{\sigma}^2\sigma^2\frac{\ve D}{\xi_\sigma}).
\end{equation}
and (\ref{e:27}) follows.
Last we consider $J_{3,l}$,
\begin{eqnarray}\label{e:25}
J_{3,l}&&=\lambda_\ve\int_{\Omega_\ve}\phi_{\ve}^\perp\frac{\partial u_l}{\partial \tau(p_l)}\,dx\nonumber\\
&&=O(\xi_\sigma^2\sigma \frac{\ve D}{\xi_\sigma}\log\frac{\ve D}{\xi_\sigma}\lambda_\ve)=o(\xi_\sigma \lambda_\ve).
\end{eqnarray}
Thus (\ref{e:26}) follows.

\end{proof}

\section{Acknowledgements}
The  research  of J.~Wei is partially supported by NSERC of Canada.
MW thanks the Department of Mathematics at UBC for their kind hospitality.


\begin{thebibliography}{99}

\bibitem{DKW} M. del Pino, M. Kowalczyk and J.C.  Wei, Multi-bump ground states of the Gierer-Meinhardt system in $R^2$, Ann. Inst. H. Poincar\'{e} Anal. Non Lin\'{e}aire 20 (2003), 53–85.

\bibitem{DGK}A. Doelman, R.A. Gardner and T.J.Kaper, Large stable pulse solutions in reaction diffusion equations, {\em Indiana Univ. Math. J.} 49(4), 2000.

\bibitem{ei-ishimoto-2013} S.-I. Ei and T. Ishimoto, Dynamics and interactions of spikes on smoothly curved
boundaries for reaction–diffusion systems in 2D, {\em Japan J. Indust. Appl. Math.} 30 (2013), 69-90.

\bibitem{ei-wei-2002} S.-I. Ei and J.C. Wei,
Dynamics of Metastable Localized Patterns and
Its Application to the Interaction of Spike
Solutions for the Gierer-Meinhardt Systems
in Two Spatial Dimensions, {\em Japan J. Indust. Appl. Math.} 19 (2002), 181-226.

\bibitem{GM}A. Gierer and H. Meinhardt, A theory of biological pattern formation, {\em Kybernetik} (Berlin) 12(1972), 30-39.

\bibitem{gt} D. Gilbarg and S. Trudinger, {\em Elliptic Partial Differential Equations
of Second Order},
Die Grundlehren der mathematischen Wissenschaften  in Einzeldarstellungen,
Vol. 224, Springer-Verlag
Berlin Heidelberg, 2nd Edition, 1983.

\bibitem{IWW}D. Iron, M.J. Ward and J. Wei, The stability of spike solutions to the one-dimensional Gierer-Meinhardt model, {\em Phys. D}, 150(2001), 25-62.

\bibitem{K}J.P. Keener, Activators and inhivitors in pattern formation, {\em Stud. Appl. Math. } 59(1978), 1-23.

\bibitem{LN}C.S. Lin and W.M. Ni, On the diffusion coefficients of a semilinear Neumann problem, {\em Calc. Var. Par. Diff. Eqns.}(Trento, 1986), 160-174.

\bibitem{M}H. Meinhardt, Models of biological pattern formation, {\em Academic Press}, London, 1982.

\bibitem{NT}W.M. Ni and I. Takagi, On the shape of least energy solution to a semilinear Neumann problems, {\em Comm. Pure Appl. Math.} 41(1991), 819-852.

\bibitem{NT1}W.M. Ni and I. Takagi, Locating the peaks of least energy solutions to a semilinear Neumann problem, {\em Duke Math. J. } 70(1993), 247-281.

\bibitem{T1}I. Takagi, Point condensation for a reaction diffusion system, {\em J.Differential Eqns.} 61(1986), 208-249.

\bibitem{T}A.M. Turing, The chemical basis of morphogenesis, {\em Phil. Trans. Roy. Soc. Lond. } B 237(1952), 37-72.

\bibitem{wmhgm}
M. J. Ward, D. McInerney, P. Houston, D. Gavaghan, P. K. Maini,
The dynamics and pinning of a spike for a reaction-diffusion system,
{\em SIAM J. Appl. Math.} 62  (2002),  1297--1328.

\bibitem{W4}J.C. Wei, On single interior spike solutions of Gierer-Meinhardt system: uniqueness and spectrum estimates, {\em Europ. J. Appl. Math.} 10(1999), 353-378.

\bibitem{W5}J.C. Wei, Uniqueness and critical spectrum of boundary spike solutions, {\em Proc. Royal Soc. Edinburgh, }Section A (Mathematics) 2001.

\bibitem{WW}J.C. Wei and M. Winter, Stationary solutions for the Cahn-Hilliard equation, {\em Ann. Inst. H. Poincare Anal. Non Lineaire}, 15(1998), 459-492.

\bibitem{WW1}J.C. Wei and M. Winter, Multiple boundary spike solutions for a wide class of singular perturbation problems, {\em J.London Math. Soc.}, 59(1998), 459-492.

\bibitem{WW2}J.C. Wei and M. Winter, On the two dimensional Gierer-Meinhardt system with strong coupling, {\em SIAM J. Math. Anal. } 30(1999), 585-606.

\bibitem{WW4} J.C. Wei, M. Winter, Spikes for the two-dimensional Gierer-Meinhardt system: the weak coupling case, {\em J. Nonlinear Science} 6(2001) 415-458.

\bibitem{WW5} J.C. Wei, M. Winter, On multiple spike solutions for the two-dimensional Gierer-Meinhardt system: the strong coupling case, {\em J. Differential Equations} 178(2002) 478-518.

\bibitem{WW3}J.C. Wei and M. Winter, Existence, classification and stability analysis of multiple-peaked solutions for the Gierer-Meinhardt system in $R^1$, {\em Methods Appl. Anal.} 14(2007), no. 2, 119-163.

\bibitem{WW6} J.C. Wei, M. Winter, On the Gierer-Meinhardt system with precursors. {\em Discrete Contin. Dyn. Syst.} 25(2009), no. 1, 363-398.

\bibitem{WW7}  J.C. Wei and M. Winter, Stability  of cluster solutions in a cooperative consumer chain model, {\it J. Math. Biol.} 68(2014), 1--39.

\bibitem{ww-book} J.C. Wei and M. Winter, {\em Mathematical Aspects of Pattern Formation in Biological Systems}, Applied Mathematical Sciences, Vol. 189, Springer, London, 2014.

\bibitem{WW8} J.C. Wei and M.Winter, Stable spike clusters for the one-dimensional Gierer-Meinhardt system, {\em Euro. J.  Applied Mathematics}, 2016.
\end{thebibliography}
\end{document}